\documentclass[12pt]{article}

\usepackage{amssymb, amsfonts, amsmath}
\usepackage{mathrsfs}
\usepackage{graphicx}
\usepackage[usenames,dvipsnames]{xcolor}
\usepackage{tikz}
\usepackage{float}
\usepackage{enumitem}
\usepackage{comment}
\usepackage{setspace}
\usepackage{caption}
\usepackage{subcaption}
\usepackage[hidelinks]{hyperref}
\usepackage{array}%
\usepackage{multicol}

\newcolumntype{C}[1]{>{\centering\let\newline\\\arraybackslash\hspace{0pt}}m{#1}}
\providecommand{\U}[1]{\protect\rule{.1in}{.1in}}
\providecommand{\boksie}{\ensuremath{\mathbin{\raisebox{0.3mm}{$\scriptstyle\square$}}}}

\newcommand{\Mod}[1]{\ (\text{mod}\ #1)}
\newcommand{\ig}[1]{\mathscr{I}(#1)}

\newcommand{\thet}[1]{\Theta\left\langle #1\right\rangle }

\providecommand{\rad}{\operatorname{rad}}

\newcommand{\Dia}{\mathfrak{D}}
\newcommand{\house}{\mathcal{H}}

\def\ieInner(#1,#2,#3,#4){#1 \overset{#2 #3}{\sim} #4}
\def\edge#1{\ieInner(#1)}
\def\ieInnerA(#1,#2,#3){\overset{#1 #2}{\sim} #3}

\def\sedge(#1,#2){#1\sim #2}
\def\ieInnerB(#1,#2,#3,#4,#5){#1 \overset{#2 #3}{\sim_{#5}} #4} 
\def\edgeG#1{\ieInnerB(#1)}
\def\ieInnerC(#1,#2,#3){#1 {\sim_{#3}} #2 }
\def\adjG#1{\ieInnerC(#1)}
\def\ieInnerD(#1,#2,#3,#4){#1 \sim #2 \sim \dots \sim #3 \sim #4}
\def\adjL#1{\ieInnerD(#1)}
\def\ieInnerE(#1,#2){#1 \sim #2}
\def\onlyedge#1{\ieInnerE(#1)}

\def\ieInnerF(#1,#2){\overset{#1 #2}{\sim}}

\definecolor{ltsky}{RGB}{0,191,255}
\definecolor{ltteal}{RGB}{0, 128, 128 }
\definecolor{medOrch}{RGB}{122,55,139}
\definecolor{royalBlue}{RGB}{65,105,225}
\definecolor{forGreen}{RGB}{34,139,34}
\definecolor{dand}{RGB}{255,193,37}
\definecolor{lightBlue}{RGB}{176,226,255}
\definecolor{lgrey}{RGB}{209,209,209}
\definecolor{lgray}{gray}{0.95}
\definecolor{mgray}{gray}{0.40}
\definecolor{mmgray}{gray}{0.60}
\definecolor{zelim}{RGB}{7,163,82}
\definecolor{zelim2}{RGB}{5,114,57}

\usetikzlibrary{fit,positioning,calc, arrows, shapes}
\usetikzlibrary{decorations.pathreplacing}
\usetikzlibrary{backgrounds}
\tikzstyle{std}=[ circle, draw=black,fill=black,thick, inner sep=2pt, minimum size=2.5mm]
\tikzstyle{wstd}=[ circle, draw=black,fill=white,thick, inner sep=2pt, minimum size=2.5mm]
\tikzstyle{ir}=[ circle, draw=black,fill=green,thick,  inner sep=2pt, minimum size=2mm]
\tikzstyle{mp}=[circle, draw=black,fill=Dandelion,thick,  inner sep=2pt, minimum size=2mm]
\tikzstyle{bred}=[circle, draw=black,fill=red,thick,  inner sep=2pt, minimum size=2mm]
\tikzstyle{byellow}=[circle, draw=black,fill=yellow,very thick,  inner sep=2pt, minimum size=3mm]
\tikzstyle{bgray}=[circle, draw=black,fill=mmgray,thick,  inner sep=2pt, minimum size=2mm]
\tikzstyle{bzed}=[circle, draw=black,fill=zelim,thick,  inner sep=2pt, minimum size=2mm]
\tikzstyle{smred}=[ circle, draw=black,fill=red,thick,  inner sep=1pt, minimum size=1.5mm]
\tikzstyle{bblue}=[ circle, draw=black,fill=blue,thick,  inner sep=2pt, minimum size=2.5mm]
\tikzstyle{bgreen}=[ circle, draw=black,fill=green,thick,  inner sep=2pt, minimum size=2.5mm]
\tikzstyle{regRed}=[ circle, draw=black,fill=red,thick,  inner sep=2pt, minimum size=2.5mm]
\tikzstyle{sqRed}=[rectangle, draw=black,fill=red,thick,  inner sep=2pt, minimum size=2.5mm]
\tikzstyle{regG}=[ circle, draw=black,fill=green,thick,  inner sep=2pt, minimum size=2mm]
\tikzstyle{dred}=[diamond, draw=black,fill=red,thick,  inner sep=2pt, minimum size=2mm]
\tikzstyle{tr}=[color=black, style=dotted]
\tikzstyle{sp}=[color=ProcessBlue, line width = 2pt]

\tikzstyle{bline}=[color=blue, line width = 1pt]
\tikzstyle{rline}=[color=red, line width = 1pt]
\tikzstyle{rlinelt}=[color=red, line width = 1pt,opacity=0.2]
\tikzstyle{rlinemb}=[color=red, line width = 1pt, densely dashdotted]
\tikzstyle{gline}=[color=mmgray, line width = 1pt]
\tikzstyle{zline}=[color=zelim, line width = 1pt]
\tikzstyle{ll}=[color=gray]

\tikzstyle{vertex}=[circle, fill=black, inner sep= 0, minimum size = 4]

\tikzset{broadArrow/.style={single arrow, fill=red!50, anchor=base, align=center,text width=2.8cm}}

\pgfdeclarelayer{blob}    
\pgfdeclarelayer{bedge}    
\pgfsetlayers{bedge,blob,main}

\newcommand{\convexpath}[2]{
	[   
	create hullcoords/.code={
		\global\edef\namelist{#1}
		\foreach [count=\counter] \nodename in \namelist {
			\global\edef\numberofnodes{\counter}
			\coordinate (hullcoord\counter) at (\nodename);
		}
		\coordinate (hullcoord0) at (hullcoord\numberofnodes);
		\pgfmathtruncatemacro\lastnumber{\numberofnodes+1}
		\coordinate (hullcoord\lastnumber) at (hullcoord1);
	},
	create hullcoords
	]
	($(hullcoord1)!#2!-90:(hullcoord0)$)
	\foreach [
	evaluate=\currentnode as \previousnode using \currentnode-1,
	evaluate=\currentnode as \nextnode using \currentnode+1
	] \currentnode in {1,...,\numberofnodes} {
		let \p1 = ($(hullcoord\currentnode) - (hullcoord\previousnode)$),
		\n1 = {atan2(\y1,\x1) + 90},
		\p2 = ($(hullcoord\nextnode) - (hullcoord\currentnode)$),
		\n2 = {atan2(\y2,\x2) + 90},
		\n{delta} = {Mod(\n2-\n1,360) - 360}
		in 
		{arc [start angle=\n1, delta angle=\n{delta}, radius=#2]}
		-- ($(hullcoord\nextnode)!#2!-90:(hullcoord\currentnode)$) 
	}
}

\usepackage{authblk}

\newtheorem{theorem}{Theorem}[section]

\newtheorem{cons}[theorem]{Construction}
\newtheorem{coro}[theorem]{Corollary}

\newtheorem{lemma}[theorem]{Lemma}

\newtheorem{obs}[theorem]{Observation}

\newtheorem{prop}[theorem]{Proposition}
\newtheorem{ques}{Question}

\newenvironment{proof}[1][Proof]{\noindent\textbf{#1.} }{\ \hfill \rule{0.5em}{0.5em}}

\begin{document}

\title{\textbf{Reconfiguration of Minimum Independent Dominating Sets in Graphs}}


\author{R. C. Brewster\thanks{Funded by Discovery Grants from the Natural Sciences and
Engineering Research Council of Canada, RGPIN-2014-04760, RGPIN-03930-2020.}}
\affil{Department of Mathematics and Statistics\\Thompson Rivers University\\805 TRU Way\\Kamloops, B.C.\\ \textsc{Canada} V2C 0C8}
\author{C. M. Mynhardt$^{*}$}
\author{L. E. Teshima}
\affil{Department of Mathematics and Statistics\\University of Victoria\\PO BOX 1700 STN CSC\\Victoria, B.C.\\\textsc{Canada} V8W 2Y2 \authorcr {\small rbrewster@tru.ca, kieka@uvic.ca, lteshima@uvic.ca}}

\maketitle

\begin{abstract}
The independent domination number $i(G)$ of a graph $G$ is the minimum cardinality of a maximal independent 
set of $G$, also called an $i(G)$-set.  
The $i$-graph of $G$, denoted $\ig{G}$, is the graph whose vertices correspond to the $i(G)$-sets, and where two 
$i(G)$-sets are adjacent if and only if they differ by two adjacent vertices.  
We show that not all graphs are $i$-graph realizable, that is, given a target graph $H$, there does not 
necessarily exist a source graph $G$ such that $H \cong \ig{G}$.  Examples of such graphs include 
$K_{4}-e$ and $K_{2,3}$. 
We build a series of tools to show that known $i$-graphs can be used to construct new $i$-graphs and apply these 
results to build other classes of $i$-graphs, such as block graphs, hypercubes, forests, cacti, 
and unicyclic graphs.
\end{abstract}

\noindent\textbf{Keywords:\hspace{0.1in}}independent domination number, graph reconfiguration, $i$-graph

\noindent\textbf{AMS Subject Classification Number 2020:\hspace{0.1in}}05C69	

\section{Introduction}
The $i$-graph $H$ of a graph $G$ is an example of a ``reconfiguration graph''. It has as its vertex set the minimum independent dominating sets of $G$, and two vertices of $H$ are adjacent whenever the symmetric difference of their corresponding sets consists of two vertices that are adjacent in $G$. We consider the following realizability question: for which graphs $H$ does there exist a graph $G$ such that $H$ is the $i$-graph of $G$? 

Following definitions and general discussions in the remainder of this section, we begin our investigation into $i$-graph realizability in Section \ref{sec:i:obs} by composing a series of observations and technical lemmas concerning the adjacency of vertices in an $i$-graph and the structure of their associated $i$-sets in the seed graph.    In Section \ref{sec:i:Real}, we present the three smallest graphs which are not $i$-graphs,  and in Section \ref{sec:i:basicGraphs}, we show that several common graph classes, like trees and cycles, are $i$-graphs.  We conclude by examining, in Section \ref{sec:i:tools},  how new $i$-graphs can be constructed from known ones.

\subsection{Reconfiguration} \label{Rec}

In general, a reconfiguration problem asks whether it is possible to transform a given source (or
seed) solution to a given problem into a target solution through a series of incremental transformations (called
reconfiguration steps) under some specified rule, such that each intermediate step is also
a solution. The resulting chain of the source solution, intermediate solutions, and target
 solution is a reconfiguration sequence.

In graph theory, reconfiguration problems are often concerned with solutions that are
vertex/edge subsets or labellings of a graph. In particular, when the solution is a vertex (or edge)
subset, the reconfiguration problem can be viewed as a token manipulation problem, where
a solution subset is represented by placing a token at each vertex or edge of the subset.
The reconfiguration step for vertex subsets can be of one of three variants (edge
subsets are handled analogously):

\begin{itemize}
\item[$\vartriangleright $] \textbf{Token Slide (TS) Model}: A single token
is slid along an edge between adjacent vertices.

\item[$\vartriangleright $] \textbf{Token Jump (TJ) Model}: A single token
jumps from one vertex to another (without the vertices necessarily being
adjacent).

\item[$\vartriangleright $] \textbf{Token Addition/Removal (TAR) Model:} A
single token can either be added to a vertex or be removed from a vertex.
\end{itemize}

To represent the many possible solutions in a reconfiguration problem, each solution can
be represented as a vertex of a new graph, referred to as a \textit{reconfiguration graph}, where adjacency between vertices follows one of the three token adjacency models, producing the
\textit{slide graph}, the \textit{jump graph}, or the \textit{TAR graph}, respectively.

More formally, given a graph $G$, the \textit{slide graph} of $G$ under some specified reconfiguration
rule is the graph $H$ such that each vertex of $H$ represents a solution of some problem
on $G$, and two vertices $u$ and $v$ of $H$ are adjacent if and only if the solution in $G$ corresponding
to $u$ can be transformed into the solution corresponding to $v$ by sliding a single
token along an edge of $G$.

\subsection{$\gamma$-Graphs}
We use the standard notation of $\gamma (G)$ for the cardinality of a
minimum dominating set of a graph $G$. The \emph{private neighbourhood} of a vertex $v$ with respect to a vertex
set $S$ is the set $\mathrm{pn}(v,S)=N[v]-N[S-\{v\}]$; therefore, a
dominating set $S$ is minimal dominating if, for each $u\in S$, $\mathrm{pn}%
(u,S)$ is nonempty. The \emph{external private neighbourhood} of $v$ with
respect to $S$ is the set $\mathrm{epn}(v,S)=\mathrm{pn}(v,S)-\{v\}$.
The \textit{independent domination number} $i(G)$ of $G$ is the
minimum cardinality of a maximal independent set of $G$, or, equivalently,
the minimum cardinality of an independent domination set of $G$. An independent dominating set of $G$ of cardinality $i(G)$ is also called an $i$-\emph{set} of $G$, or an $i(G)$-\emph{set}.

In general, we follow the notation of \cite{CLZ}. In particular, the disjoint union of two
graphs $G$ and $H$ is denoted $G\cup H$, whereas the join of $G$ and $H$,
denoted $G\vee H$, is the graph obtained from $G\cup H$ by joining every
vertex of $G$ with every vertex of $H$. For other domination principles and
terminology, see \cite{HHS2, HHS1}.

First defined by Fricke, Hedetniemi, Hedetniemi, and Hutson \cite{FHHH11} in 2011,
the $\gamma $-\textit{graph of a graph} $G$ is the graph $G(\gamma
)=(V(G(\gamma )),E(G(\gamma )))$, where each vertex $v\in V(G(\gamma ))$
corresponds to a $\gamma $-set $S_{v}$ of $G$. The vertices $u$ and $v$ in $%
G(\gamma )$ are adjacent if and only if there exist vertices $u^{\prime }$
and $v^{\prime }$ in $G$ such that $u^{\prime }v^{\prime }\in E(G)$ and $%
S_{v}=(S_{u}-u^{\prime })\cup \{v^{\prime }\}$;  this is a token-slide model
of adjacency. 

An initial question of Fricke et al. \cite{FHHH11} was to determine exactly which graphs
are $\gamma $-graphs; they showed that every tree is the $\gamma $%
-graph of some graph and conjectured that every graph is the $\gamma $-graph of some graph. Later that year,
Connelly, Hutson, and Hedetniemi \cite{CHH10} proved this conjecture to be true. For additional results on $\gamma $-graphs, see \cite{B15, CHH10, EdThesis, FHHH11}. Mynhardt and Teshima \cite{MT18} investigated slide model reconfiguration graphs with respect to other domination parameters.

Subramanian and Sridharan \cite{SS08} independently defined a different $\gamma $-%
\textit{graph of a graph} $G$, denoted $\gamma \cdot G$. The vertex set of $%
\gamma \cdot G$ is the same as that of $G(\gamma )$; however, for $u,w\in
V(\gamma \cdot G)$ with associated $\gamma $-sets $S_{u}$ and $S_{w}$ in $G$%
, $u$ and $w$ are adjacent in $\gamma \cdot G$ if and only if there exist
some $v_{u}\in S_{u}$ and $v_{w}\in S_{w}$ such that $S_{w}=(S_{u}-\{v_{u}%
\})\cup \{v_{w}\}$. This version of the $\gamma $-graph was dubbed the
\textquotedblleft single vertex replacement adjacency
model\textquotedblright\ by Edwards \cite{EdThesis}, and is sometimes referred to as
the \textquotedblleft jump $\gamma $-graph\textquotedblright\ as it follows
the TJ-Model for token reconfiguration. Further results concerning $\gamma
\cdot G$ can be found in \cite{LV10, SS13, SS09}. Notably, if $G$ is a tree or a
unicyclic graph, then there exists a graph $H$ such that $\gamma \cdot H=G$ \cite{SS09}. Conversely, if $G$ is the 
(jump) $\gamma $-graph of some graph $H$,
then $G$ does not contain any induced $K_{2,3},\ P_{3}\vee K_{2}$, or $%
(K_{1}\cup K_{2})\vee 2K_{1}$ \cite{LV10}.

Using a token addition/removal model, Haas and Seyffarth \cite{HS14} define the $k$-%
\textit{dominating graph} $D_{k}(G)$ of $G$ as the graph with vertices
corresponding to the $k$-dominating sets of $G$ (i.e., the dominating sets
of cardinality at most $k$). Two vertices in the $k$-dominating graph are
adjacent if and only if the symmetric difference of their associated $k$%
-dominating sets contains exactly one element. Additional results can be
found in \cite{Adaricheva2021, AFK17, HS17, HIMNOST16, SMN16}, and a survey on reconfiguration of colourings and dominating sets of graphs in \cite{MN20}.

\subsection{$i$-Graphs}
The \emph{$i$-graph} of a graph $G$, denoted $\ig{G}=(V(\ig{G}),E(\ig{G}))$, is the graph with vertices 
representing the minimum independent dominating sets of $G$ (that is, the \emph{$i$-sets} of $G$).  
As in the case of $\gamma $-graphs as defined in \cite{FHHH11}, adjacency in $\ig{G}$ follows a slide
model where $u,v\in V(\ig{G})$, corresponding to the $i(G)$-sets  $S_u$ and $S_v$, respectively, are 
adjacent in $\ig{G}$ if and only if there exists  $xy \in E(G)$ such that $S_u = (S_v-x)\cup\{y\}$.  
We say $H$ \emph{is an $i$-graph}, or is $i$-\emph{graph realizable}, if there exists some graph $G$ such that 
$\ig{G}\cong H$. Moreover, we refer to $G$ as the \emph{seed graph} of the $i$-graph $H$.  
Going forward, we mildly abuse notation to denote both the $i$-set $X$ of $G$ and its corresponding 
vertex in $H$ as $X$, so that $X \subseteq V(G)$ and $X \in V(H)$. 

Imagine that there is a token on each vertex of an $i$-set $S$ of $G$.  Then $S$ is adjacent, in $\ig{G}$, to an $i(G)$-set $S'$ if and only if a single token can be slid along an edge of $G$ to transform $S$ into $S'$.  Notice that the \emph{token jump model} of reconfiguration for independent domination is identical to the token-slide model.  On a graph $G$ a token may only ``jump'' from a vertex $v$ in the $i$-set $S_1$ to another vertex $w$ (to form the $i$-set $S_2$) if $w$ is dominated only by $v$ in $S_1$.  Otherwise,  if $w$ is dominated by some other $u \neq v$ in $S_1$, then $(S_1-{u}) \cup \{w\}$ is not an independent set as it contains the adjacent vertices $u$ and  $w$.  A token is said to be \emph{frozen} (in any reconfiguration model) if there are no available vertices to which it can slide/jump.

In acknowledgment of the slide-action in $i$-graphs, given $i$-sets  $X = \{x_1,x_2,\dots,x_k\}$ and 
$Y=\{y_1,x_2,\dots x_k\}$ of $G$ with $x_1y_1 \in E(G)$, we denote the adjacency of $X$ and $Y$ in $\ig{G}$ 
as $\edge{X,x_1,y_1,Y}$, where we imagine transforming the $i$-set $X$ into $Y$ by sliding the token at $x_1$ 
along an edge to $y_1$.  When discussing several graphs, we use the notation $\edgeG{X,x_1,y_1,Y,G}$ 
to specify that the relationship is on $G$.  
More generally, we use $x \sim y$ to denote the adjacency of vertices $x$ and $y$ (and $x\not \sim y$ to 
denote non-adjacency); this is used in the context of both the seed graph and the target graph.

Although every graph is the $\gamma $-graph of some graph, there is no such  tidy theorem for $i$-graphs; as we show in Section \ref{sec:i:Real}, not every graph is an $i$-graph, and determining which classes of graphs are (or are not) $i$-graphs has proven to be an interesting challenge.

\section{Observations} \label{sec:i:obs}

To begin, we propose several observations about the structure of $i$-sets within given $i$-graphs which we then use to construct a series of useful lemmas.

\begin{obs} \label{obs:i:edge}
	Let $G$ be a graph and $H=\ig{G}$.  A vertex $X \in V(H)$ has $\deg_H(X)\geq 1$ if and only if for some $v \in X \subseteq V(G)$, there exists $u\in \mathrm{epn}(v,X)$ such that $u$ dominates $\mathrm{pn}(v,X)$.  
\end{obs}

\noindent From a token-sliding perspective, Observation \ref{obs:i:edge} shows that a token on an $i$-set vertex $v$ is frozen if and only if $\mathrm{epn}(v) = \varnothing$ or $G[\mathrm{epn}(v,X)]$ has no dominating vertex.

For some path $X_1,X_2,\dots,X_k$ in $H$, only one vertex of the $i$-set is changed at each step, and so $X_1$ and $X_k$ differ on at most $k$ vertices.  This yields the following observation.  

\begin{obs} \label{obs:i:dist}
	Let $G$ be a graph and $H = \ig{G}$.  Then for any $i$-sets $X$ and $Y$ of $G$, the distance $d_H(X,Y) \geq |X-Y|$.  
\end{obs}


\begin{lemma} \label{lem:i:claw}
	Let $G$ be a graph with $H = \ig{G}$.  Suppose $XY$ and $YZ$ are edges in $H$ with $\edge{X,x,y_1,Y}$ and $\edge{Y,y_2,z,Z}$, with $X \neq Z$.  Then $XZ$ is an edge of $H$ if and only if $y_1=y_2$.
\end{lemma}

\begin{proof}
	Let $X=\{x,v_2,v_3\dots,v_k\}$ and $Y=\{y_1,v_2,,v_3\dots,v_k\}$ so that $\edge{X,x,y_1,Y}$.  	
	To begin, suppose $y_1=y_2$.  Then  $\edge{Y,y_1,z,Z}$ and $Z=\{z,v_2,,v_3\dots,v_k\}$, hence $|X-Z|  = 1$.   Since $X$ is dominating, $z$ is adjacent to a vertex in $\{x,v_2,v_3\dots,v_k\}$; moreover, since $Z$ is independent, $z$ is not adjacent to any of $\{v_2,v_3,\dots v_k\}$.  Thus $z$ is adjacent to $x$ in $G$ and $\edge{X,x,z,Z}$, so that $XZ \in E(H)$.  
	
	Conversely, suppose $y_1\neq y_2$.  Then, without loss of generality, say $y_2=v_2$ and so $X=\{x,y_2,v_3\dots,v_k\}$, $Y=\{y_1,y_2,v_3\dots,v_k\}$, and $Z=\{y_1,z,v_3\dots,v_k\}$.  Notice that $x \neq z$ since $x \sim y_1$ and $z \not\sim y_1$.  Thus  $|X-Z| = 2$, and it follows that $XZ \notin E(H)$.
\end{proof}

\medskip

\noindent Combining Observation \ref{obs:i:dist} and Lemma \ref{lem:i:claw} yields the following observation for vertices of $i$-graphs at distance two. 

\begin{obs} \label{obs:i:d2}
	Let $G$ be a graph and $H = \ig{G}$.  Then for any $i$-sets $X$ and $Y$ of $G$, if $d_H(X,Y) = 2$, then $|X-Y|=2$.  
\end{obs}

\begin{lemma} \label{lem:i:K1m}  Let $G$ be a graph and  $H = \ig{G}$.  Suppose $H$ contains an induced $K_{1,m}$ with vertex set $\{X,Y_1,Y_2, \dots, Y_m\}$ and $\deg_H(X)=m$.  Let $i \neq j$.  Then in $G$,  
	\begin{enumerate}[itemsep=1pt, label=(\roman*)]
		\item $X-Y_i \neq X-Y_j$, \label{lem:K1m:a}
		\item  $|Y_i \cap Y_j| = i(G)-2$, and  \label{lem:K1m:b}
		\item $m \leq i(G)$. \label{lem:K1m:c}
	\end{enumerate}
\end{lemma}

\begin{proof}
	Suppose $\edge{X,x_i,y_i,Y_i}$ and $\edge{X,x_j,y_j,Y_j}$.  Then $(X - Y_i) = \{x_i\}$ and $(X-Y_j) = \{x_j\}$.    From Lemma \ref{lem:i:claw}, since $Y_i \not\sim Y_j$, we have that $x_i \neq x_j$, which establishes Statement \ref{lem:K1m:a}.  Moreover, $Y_i \cap Y_j = X - \{x_i,x_j\}$, and so as these are $i$-sets, Statement \ref{lem:K1m:b} also follows.  Finally, for Statement \ref{lem:K1m:c},  again applying Lemma \ref{lem:i:claw}, we see that $|\bigcap_{1\leq i \leq m} Y_i| = |X|-m = i(G)-m \geq 0$.	
\end{proof}

\medskip

Induced $C_4$'s in a target graph $H$ play an important role in determining the $i$-graph realizability of $H$ and determine a specific relationship among $i$-sets of a potential source graph $G$, as we show next.

\begin{prop} \label{prop:i:C4struct}
	Let $G$ be a graph and $H=\ig{G}$.  Suppose $H$ has an induced $C_4$ with vertices $X,A,B,Y$, where $XY, AB \notin E(H)$.   Then, without loss of generality, the set composition of $X,A,B,Y$ in $G$, and the edge labelling  of  the induced $C_4$ in $H$, are as in Figure \ref{fig:i:C4struct}.
\end{prop}

\begin{figure}[H]
	\centering
	\begin{tikzpicture}[scale=0.8]
		
		\node[std,label={270:$Y=\{y_1,y_2,v_3,\dots,v_k\}$}] (y) at (0,0) {};
		\node[std,label={180:$A=\{y_1,x_2,v_3,\dots,v_k\}$}] (a) at (-2,2) {};
		\node[std,label={0:$B=\{x_1,y_2,v_3,\dots,v_k\}$}] (b) at (2,2) {};
		\node[std,label={90:$X=\{x_1,x_2,v_3,\dots,v_k\}$}] (x) at (0,4) {};
		
		\draw[thick] (x)--(a) node[midway,left=10pt,fill=white]{$\edge{X,x_1,y_1,A}$};
		\draw[thick] (x)--(b) node[midway,right=10pt,fill=white]{$\edge{X,x_2,y_2,B}$};
		\draw[thick] (a)--(y) node[midway,left=10pt,fill=white]{$\edge{A,x_2,y_2,Y}$};
		\draw[thick] (b)--(y) node[midway,right=10pt,fill=white]{$\edge{B,x_1,y_1,Y}$};
	\end{tikzpicture}			
	\caption{Reconfiguration structure of an induced $C_4$ subgraph from Proposition \ref{prop:i:C4struct}.}
	\label{fig:i:C4struct}%
\end{figure}
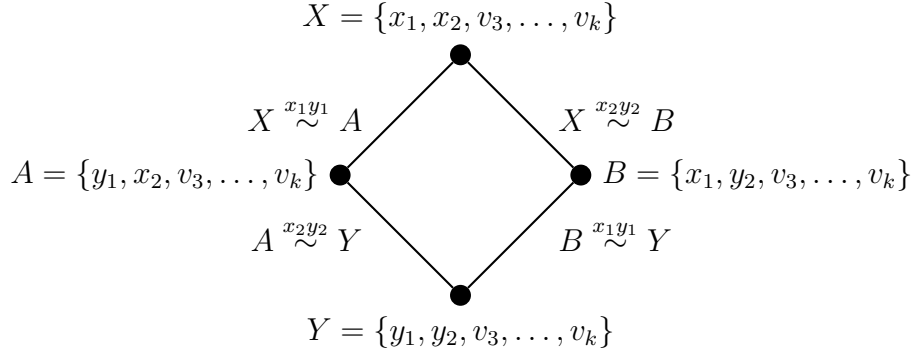

\begin{proof}
	Suppose that the $i$-set $X$ of $G$ has $X=\{x_1,x_2,v_3,\dots,v_k\}$. 
	Then by Lemma \ref{lem:i:claw}, without loss of generality, the edge from $X$ to $A$ can be labelled as $\edge{X,x_1,y_1,A}$ for some $y_1 \in V(G) - X$, so that $A = \{y_1,x_2,v_3,\dots,v_k\}$, while the edge from $X$ to $B$ can be labelled $\edge{X,x_2,y_2,B}$ for some $y_2$ and $B=\{x_1,y_2,v_3,\dots,v_k\}$.  
	
	Consider the edge $AY \in E(H)$ labelled $\edge{A,a,y^*,Y}$.  From Lemma \ref{lem:i:claw}, since  $XY \notin E(G)$, $a \neq y_1$.  If, say, $a=v_3$, then $Y=\{y_1,x_2,y^*,\dots,v_k\}$.  However, neither $y_1$ nor $x_2$ is in $B$, so $|Y-B| \geq 2$, contradicting Observation \ref{obs:i:dist}.  	
	Thus, $a \neq v_i$ for any $3 \leq i \leq k$.  This leaves $a=x_2$, and $Y=\{y_1,y^*,v_3\dots,v_k\}$.  Since $|Y-B|=1$, $y^*=y_2$ and $Y=\{y_1,y_2,v_3,\dots,v_k\}$   as required.
\end{proof}

\section{Realizability of $i$-Graphs} \label{sec:i:Real}

Having now established  a series of observations and lemmas about the structures of $i$-graphs and the composition of their associate $i$-sets, we demonstrate that not all graphs are $i$-graphs by presenting three counterexamples: the diamond graph $\Dia$, $K_{2,3}$ and $\kappa$, as pictured in Figure \ref{fig:3nonReal}.

\begin{figure}[H]
	\centering
	\begin{figure}[H]
		\centering
		\begin{tikzpicture}
			
			\node(cent) at (0,0) {};			
			\path (cent) ++(180:35 mm) node(lcent)  {};	
			
			\path (lcent) ++(90:10 mm) node[std](x)  {};	
			\path (x) ++(90: 4mm) node (xL)  {$X$};
			\path (lcent) ++(-90:10 mm) node[std](y)  {};
			\path (y) ++(-90: 4mm) node (zL)  {$Y$};
			\path (lcent) ++(180:8 mm) node[std](a)  {};
			\path (a) ++(180: 4mm) node (aL)  {$A$};
			\path (lcent) ++(0:8 mm) node[std](b)  {};
			\path (b) ++(0: 4mm) node (bL)  {$B$};
			
			\draw(x)--(y)--(a)--(x);
			\draw(x)--(b)--(y);
			
			\path (lcent) ++(-90: 24mm) node (GL)  {$\Dia=K_4-e$};	
			
			
			\path (cent) ++(90:10mm) node[std](x)  {};
			\path (x) ++(90: 4mm) node (xL)  {$X$};
			\path (cent) ++(-90:10 mm) node[std](y)  {};
			\path (y) ++(-90: 4mm) node (yL)  {$Y$};
			\path (cent) ++(180:8 mm) node[std](a)  {};
			\path (a) ++(180: 4mm) node (aL)  {$A$};
			\path (cent) ++(0:8 mm) node[std](c)  {};
			\path (c) ++(0: 4mm) node (cL)  {$C$};
			\path (cent) ++(0:0 mm) node[std](b)  {};
			\path (b) ++(0: 4mm) node (bL)  {$B$};
			
			\draw(x)--(a)--(y)--(c)--(x)--(b)--(y);
			
			\path (cent) ++(-90: 24mm) node (GL)  {$K_{2,3}$};			
			
			
			\path (cent) ++(0:35 mm) node[std](brcent)  {};	
			\path (brcent) ++(0: 4mm) node (bL)  {$B$};			
			
			\path (brcent) ++(90:10mm) node[std](x)  {};	
			\path (x) ++(90: 4mm) node (xL)  {$X$};
			\path (brcent) ++(-90:10 mm) node[std](y)  {};
			\path (y) ++(-90: 4mm) node (yL)  {$Y$};
			\path (brcent) ++(180:8 mm) node[std](a)  {};
			\path (a) ++(180: 4mm) node (aL)  {$A$};
			\path (brcent) ++(0:8 mm) node (b)  {};
			
			\path (b) ++(90:4 mm) node[std](b1)  {};
			\path (b1) ++(0: 4mm) node (bL)  {$C_1$};
			\path (b) ++(-90:4 mm) node[std](b2)  {};
			\path (b2) ++(0: 4mm) node (bL)  {$C_2$};
			
			\draw(x)--(y)--(a)--(x);
			\draw(x)--(b1)--(b2)--(y);

			\path (brcent) ++(-90: 24mm) node (GL)  {$\kappa: K_{2,3}$ with a};
			\path (brcent) ++(-90: 30mm) node (GL)  {subdivided edge};
		\end{tikzpicture}	
	\end{figure}			
	\caption{Three graphs not realizable as $i$-graphs.}
	\label{fig:3nonReal}%
\end{figure}

\begin{prop} \label{prop:i:diamond}
	The diamond graph $\Dia = K_4-e$ is not $i$-graph realizable.
\end{prop}

\begin{proof}
	Suppose to the contrary there is some graph $G$ with $\ig{G}=\Dia$.  Let $V(\Dia)=\{X,A,B,Y\}$ where $AB \notin E(\Dia)$.   Say $\edge{X,x,y,Y}$.  Then by Lemma \ref{lem:i:claw}, without loss of generality, the edges incident with $A$ can be labelled as $\edge{X,x,a,A}$ and $\edge{A, a,y,Y}$.  Likewise, $\edge{X,x,b,B}$ and $\edge{B, b,y,Y}$ (see Figure \ref{fig:3nonReal}).  However,  since $\edge{B, b,x,X}$ and $\edge{X,x,a,A}$, Lemma \ref{lem:i:claw} implies that $AB\in E(\Dia)$, a contradiction.	  
\end{proof}

\begin{prop} \label{prop:i:K23}
	The graph $K_{2,3}$ is not $i$-graph realizable.
\end{prop}

\begin{proof}
	Suppose $K_{2,3} = \ig{G}$ for some graph $G$.  Let  $ \{\{X,Y\},\{A,B,C\}\}$ be the bipartition of $K_{2,3}$. 
	Apply the exact labelling from Proposition \ref{prop:i:C4struct} and Figure \ref{fig:i:C4struct} to the $i$-sets and edges of  $X,A,B,$ and $Y$.  We attempt to extend the labelling to $C$.  By Lemma \ref{lem:i:claw}, since $C$ is adjacent to $X$, but not $A$ or $B$,  without loss of generality, $\edge{X, v_3,c,C}$ and $C=\{x_1,x_2,c,v_4,\dots,v_k\}$.  As $A$ is an $i$-set, $y_1v_3\notin E(G)$.   Since $v_3c\in E(G)$, $c \neq y_1$.  Similarly, $c \neq y_2$.  
	Now $|C-Y| =3$ and $d(C,Y)=1$, contradicting Observation \ref{obs:i:dist}.   	
\end{proof}


\begin{prop} \label{prop:i:kappa}
	The graph $\kappa$ is not $i$-graph realizable.
\end{prop}

\begin{proof}
	Suppose $\kappa = \ig{G}$ for some graph $G$ and let $V(\kappa) = \{X,A,B,C_1,C_2,Y\}$ as in  Figure  \ref{fig:3nonReal}, and to the subgraph induced by $X,A,B,Y$, apply the labelling of Proposition \ref{prop:i:C4struct} and Figure \ref{fig:i:C4struct}.  Through additional applications of Proposition \ref{prop:i:C4struct}, we can, as in the proof of Proposition \ref{prop:i:K23}, assume without loss of generality that $\edge{X,x_3,y_3,C_1}$.  However, $d(C_1,Y)=2$ but $|Y-C_1|=3$, contradicting Observation \ref{obs:i:dist}.  It follows that no such $G$ exists and $\kappa$ is not an $i$-graph.
\end{proof}

\medskip


The observant reader will have undoubtedly noticed the common structure between the graphs in the previous three propositions -- they are all members of the class of \emph{theta graphs} \label{i:pg:theta} (see \cite {B72}), graphs that are the union of three internally disjoint nontrivial paths with the same two distinct end vertices.  The graph $\thet{j,k,\ell}$ with $j \leq k \leq \ell$, is the theta graph with paths of lengths $j$, $k$,  and $\ell$.  In this notation, our three non $i$-graph realizable examples are $\Dia \cong \thet{1,2,2}$, $K_{2,3}\cong\thet{2,2,2}$, and $\kappa \cong \thet{2,2,3}$.

Further rumination on the similarity in structure suggests that additional subdivisions of the central path in $\kappa$ could yield more theta graphs that are not $i$-graphs.  However, the proof technique used for $\kappa$ no longer applies when, for example, a path between the degree $3$ vertices has length greater than $4$.  In \cite{BMT3}, we explore an alternative method for determining the $i$-graph realizability of theta graphs.


\section{Some Classes of $i$-Graphs}
\label{sec:i:basicGraphs}

Having studied several graphs that are not $i$-graphs, we now examine the problem of $i$-graph realizability from the positive direction.  To begin, it is easy to see that complete graphs are $i$-graphs; moreover, as with $\gamma$-graphs, complete graphs are their own $i$-graphs, i.e., $\ig{K_n} \cong K_n$.  

\begin{prop} \label{prop:comp} 
	Complete graphs are $i$-graph realizable. 
\end{prop}

\noindent  Hypercubes $Q_n$ (the Cartesian product of $K_2$ taken with itself $n$ times) are also straightforward to construct as $i$-graphs, with $\ig{nK_2} \cong Q_n$.  Each $K_2$ pair can be viewed as a $0-1$ switch, with the vertex of the $i$-set in each component sliding between the two states.

\begin{prop} \label{prop:hCube} 
	Hypercubes are $i$-graph realizable. 
\end{prop}

\noindent  Hypercubes are a special case of the following result regarding Cartesian products of $i$-graphs.

\begin{prop} \label{prop:i:cart}
	If $\ig{G_1} \cong H_1$ and $\ig{G_2} \cong H_2$, then $\ig{G_1 \cup G_2} \cong H_1 \boksie H_2$
\end{prop}

\begin{proof}
	Let $\{X_1, X_2, \dots, X_k\}$ be the $i$-sets of $G_1$ and let $\{Y_1, Y_2, \dots, Y_{\ell}\}$ be the $i$-sets of $G_2$.  Then, the $i$-sets of $G_1 \cup G_2$ are of the form $X_i \cup Y_j$.  Clearly $\edgeG{X_i\cup Y_j,,,X^*_{i^*}\cup Y^*_{j^*},G_1 \cup G_2}$ if and only if $\edgeG{X_i,,,X^*_i,G}$ and $Y_j = Y_j^*$, or $\edgeG{Y_j,,,Y_j^*,G_2}$ and $X_i=X_i^*$.		This gives a natural isomorphism to $H_1 \boksie H_2$, where $X_i \cup Y_j$ is the vertex $(X_i,Y_j)$.
\end{proof}	

\medskip

\noindent Moving to cycles, the constructions become markedly more difficult.  

\begin{prop} \label{prop:i:cycle} 
	Cycles are $i$-graph realizable. 
\end{prop}

\begin{proof}	
	The constructions for each cycle  $C_k$ for $k \geq 3$ are as described below.	
	\begin{enumerate}[itemsep=1pt, label=\bf (\roman*)]
		\item  $\ig{C_3} \cong C_3$ \\
		From Proposition \ref{prop:comp}.
		
		\item  $\ig{2K_2} \cong C_4$\\			
		From Proposition \ref{prop:hCube}.

		\item $\ig{C_5} \cong C_5$\\		
		Recall that $i(C_5)=2$.  A labelled $C_5$ and the resulting $i$-graph with $\ig{C_5}\cong C_5$ are given in Figure \ref{fig:i:C5C5} below.
		
		\begin{figure}[H]
			\centering
			\begin{tikzpicture}	
				
				\node (cent) at (0,0) {};
	
				\path (cent) ++(0:0 mm) node (C1)  {};		
				
				\foreach \i/\c in {1/90/, 5/162, 4/234, 3/306, 2/378} {
					\path (C1) ++(\c:10 mm) node[std] (v\i)  {};
					\path (v\i) ++(\c:5 mm) node (vL\i) {$x_{\i}$};  
					
				}			
				
				\draw (v1)--(v2)--(v3)--(v4)--(v5) --(v1);	
				
				\path (cent) ++(0:60 mm) node (C2)  {};		
				
				\foreach \i/\c/\x/\y in {1/90/1/4, 2/162/1/3, 3/234/3/5, 4/306/2/5, 5/378/2/4} {
					\path (C2) ++(\c:10 mm) node[std] (u\i)  {};
					\path (u\i) ++(\c:8 mm) node (uL\i) {$\{x_{\x},x_{\y}\}$};  
					
				}
				
				\draw (u1)--(u2)--(u3)--(u4)--(u5) --(u1);

			\end{tikzpicture}			
			\caption{ $C_5$ and $\ig{C_5}\cong C_5$.}
			\label{fig:i:C5C5}%
		\end{figure}

		\item  $\ig{K_2 \boksie K_3} \cong C_6$		
		
		\noindent	Label the vertices of $K_2 \boksie K_3$ as in Figure \ref{fig:i:C6} below.  The set $\{x_i,y_j\}$ is an $i$-set of $K_2 \boksie K_3$ if and only if $i \neq j$, so that $|V(\ig{K_2 \boksie K_3})| = 6$, and adjacencies are as in Figure \ref{fig:i:C6}. 
		
		\begin{figure}[H]
			\centering
			\begin{tikzpicture}	
				
				\node (cent) at (0,0) {};
				
				\foreach \i/\c in {3/-10,2/-20,1/-30} {
					\path (cent) ++(\c mm, -5 mm) node[std] (x\i)  {};
					\path (x\i) ++(270:5 mm) node (xL\i) {$x_{\i}$};
					\path (cent) ++(\c mm, 5 mm) node[std] (y\i) {};
					\path (y\i) ++(90:5 mm) node (yL\i) {$y_{\i}$};
					\draw (x\i)--(y\i);				
				}
				
				\draw(x1)--(x3);
				\draw(y1)--(y3);
				
				\draw(x1) to [out=-20,in=200](x3);
				\draw(y1) to [out=20,in=160](y3);				
				
				\path (cent) ++(0:40 mm) node (C2)  {};		
				
				\foreach \i/\c/\x/\y in {1/90/1/2, 2/150/1/3, 3/210/2/3, 4/270/2/1, 5/330/3/1, 6/30/3/2} {
					\path (C2) ++(\c:10 mm) node[std] (u\i)  {};
					\path (u\i) ++(\c:8 mm) node (uL\i) {$\{x_{\x},y_{\y}\}$};  
					
				}
				
				\draw (u1)--(u2)--(u3)--(u4)--(u5)--(u6)--(u1);

			\end{tikzpicture}			
			\caption{ $K_2 \; {\scriptstyle\square} \; K_3$ and $\ig{K_2 \; {\scriptstyle\square} \; K_3} \cong C_6$.}
			\label{fig:i:C6}%
		\end{figure}

		
		\item \emph{For any $k \geq 7$, construct the graph $H$ with $V(H) = \{v_0,v_1, \dots,v_{k-1}\}$, and $v_iv_j \in E(H)$ if and only if $j \not\equiv i-2, i-1,i,i+1,i+2 \Mod k$.  Then $\ig{H} \cong C_k$.}

		For convenience, we assume that all subscripts are given modulo $k$.  Thus in $H$, for all $0 \leq i \leq k-1$, we have the following:
		\begin{enumerate}[label=(\Roman*)]
			\item $N[v_i] \backslash N[v_{i+1}] = \{v_i,v_{i+3}\}$ \label{prop:i:cycle:k:1}
			\item $N[v_{i+1}] \backslash N[v_i] = \{v_{i-2}, v_{i+1}\}$. \label{prop:i:cycle:k:2}
		\end{enumerate}
		
		 Since $H$ is vertex-transitive, suppose that $v_i$ is in some $i$-set $S$.  Then $v_{i-2},v_{i-1},v_{i+1},$ and $v_{i+2}$ are not dominated by $v_i$.  To dominate $v_{i+1}$, either $v_{i+1}$ or $v_{i-2}$ is in $S$, because all other vertices in $N[v_{i+1}]$ are also adjacent to $v_i$, as in \ref{prop:i:cycle:k:2}.  Begin by assuming that $v_{i+1} \in S$.  Now since $\{v_i,v_{i+1}\}$ dominates all of $H$ except $v_{i+2}$ and $v_{i-1}$, and $N({v_{i+2},v_{i-1}}) \subseteq N(v_i,v_{i+1})$, either $v_{i+2}$ or $v_{i-1}$ is in $S$.  Thus $S = \{v_i,v_{i+1},v_{i+2}\}$ or $S = \{v_{i-1},v_i,v_{i+1} \}$.
		
		Suppose now instead that $v_{i-2} \in S$.     Now, only $v_{i-1}$ is not dominated by $\{v_i, v_{i-2}\}$; moreover, since $N(v_{i-1}) \subseteq N(\{v_{i-2}, v_{i}\})$, we have that $v_{i-1} \in S$, and so  $S = \{v_{i-2}, v_{i-1}, v_i\}$.
		Combining the above two cases yields that $i(H) =3$ and that all $i$-sets of $H$ have the form $S_i=\{v_i, v_{i+1}, v_{i+2}\}$, for each $ 0  \leq i \leq k-1$.  Moreover, as there are $k$ unique such sets, it follows that $|V(\ig{H})| = k$.

		We now consider the adjacencies of $\ig{H}$.  From our set definitions,  $\edge{S_i, v_{i-1}, v_{i+2}, S_{i+1}}$, and $\edge{S_i, v_{i+1}, v_{i-2}, S_{i-1}}$.  To see that $S_i$ is not adjacent to any other $i$-set in $H$, notice that the token at $v_i$ is frozen; $N(v_i)) \subseteq (N(v_{i-1}) \cup N(v_{i+1}))$.  Moreover, by \ref{prop:i:cycle:k:2}, the token at $v_{i+1}$ can only slide to $v_{i-2}$, and likewise, the token at $v_{i-1}$ can only slide to $v_{i+2}$.  Thus $\adjL{S_i,S_{i+1},S_{i-1},S_{i}}$, and so $\ig{H} \cong C_k$ as required.	
	\end{enumerate}		
	This completes the $i$-graph constructions for all cycles.
\end{proof}

\medskip


The constructions presented in Proposition \ref{prop:i:cycle} are not unique.  Brewster, Mynhardt and Teshima show in \cite{BMT2} that for $k \geq 5$ and $k \equiv 2 \Mod 3$, $\ig{C_k} \cong C_k$, and in \cite{BMT3} they use graph complements to construct graphs with $i$-graphs that are cycles.   

We now present three lemmas with the eventual goal of demonstrating that all forests are $i$-graphs.  When considering the  $i$-graph of some graph  $H$, if a vertex $v$ of some $i$-set $S$ of $H$ has no external private neighbours, then the token at $v$ is frozen.  In the first of the three lemmas,  Lemma \ref{lem:epnAdd}, we construct a new seed graph for a given target graph, where each vertex of the seed graph's $i$-set has a non-empty private neighbourhood.  


\begin{lemma} \label{lem:epnAdd}
	For any graph $H$, there exists a graph $G$ such that $\ig{G} \cong \ig{H}$ and for any $i$-set $S$ of $G$, all $v\in S$ have $\mathrm{epn}(v,S) \neq \varnothing$.
\end{lemma}

\begin{proof}
	Suppose $S$ is an $i$-set of $H$ having some $v \in S$ with $\mathrm{epn}(v,S) = \varnothing$.  Construct the graph $G_1$ from $H$ by joining new vertices $a$ and $b$ to each vertex of $N[v]$.  
	
	To begin, we show that the $i$-sets of $G_1$ are exactly the $i$-sets of $H$.  Let $R$ be some $i$-set of $H$ and say that $v$ is dominated by $u\in R$.  Then $u \in N[v]$, so $u$ also dominates $a$ and $b$ in $G_1$; therefore, $R$ is independent and dominating in $G_1$, and so $i(G_1) \leq i(H)$.  Conversely, suppose that $Q$ is an $i$-set of $G_1$.  If neither $a$ nor $b$ is in $Q$, then $Q$ is an independent dominating  set of $H$, and so $i(H) \leq i(G_1)$.  Hence, suppose instead that $a \in Q$.  Notice that since $Q$ is independent and $a$ is adjacent to each vertex in $N(v)$, $N_H[v] \cap Q = \varnothing$.  Some vertex in $Q$ dominates $b$; however, since $N_H(a) = N_H(b)$ and $Q$ is independent, it follows that $b$ is self-dominating and so $b\in Q$.  	
	However, since $N_{G_1}[\{a,b\}] = N_{G_1}[v]$, the set $Q' = (Q - \{a,b\}) \cup \{v\}$ is an independent dominating set of $G_1$ such that $|Q'| < |Q|$, a contradiction.  Thus, $i(G_1) = i(H)$ and the $i$-sets of $H$ and $G_1$ are identical.  In particular, $S$  is an $i$-set of $G_1$, and moreover, $\mathrm{epn}_{G_1}(v,S) = \{a,b\}$.	
	
	By repeating the above process for each $i$-set of $G_j$, $j\geq 1$, that contains a vertex $v$ with $\mathrm{epn}(v,S) = \varnothing$, we eventually obtain a graph $G=G_k$ such that for each $i$-set of $S$ of $G$ and each vertex $v \in S$, $\mathrm{epn}_G(v,S) \neq \varnothing$.  Since the $i$-sets of $H$ and $G$ are identical and $H$ is a subgraph of $G$, $\ig{G} = \ig{H}$ as required.  	
\end{proof}

\medskip

Next on our way to constructing forests, we demonstrate that given an $i$-graph, the graph obtained by adding any number of isolated vertices is also an $i$-graph.


\begin{lemma} \label{lem:vAdd}
	If $H$ is the $i$-graph of some graph $G$, then there exists some graph $G^*$ such that $\ig{G^*} = H \cup \{v\}$.  
\end{lemma}



\begin{proof}
	First assume that $i(G) \geq 2$.  Let $V(G) = \{v_1, v_2, \dots, v_n\}$ and let $W$ be an independent set of size $i(G)=k$ disjoint from $V(G)$, say $W=\{w_1, w_2, \dots, w_k \}$.  Construct a new graph $G^*$ by taking the join of $G$ with the vertices of $W$, so that $G^* = G \vee W$.

	
	Notice that $W$ is independent and dominating in $G^*$.  Moreover, if an $i$-set $S$ of $G^*$ contains any vertex $w_i$ of $W$, since $W$ is independent and each vertex of $W$ is adjacent to all of $\{v_1, v_2, \dots, v_n\}$, it follows that $S$ contains all of $W$, and so, $S=W$.  That is, if an $i$-set of $G^*$ contains any vertex of $W$, it contains all of $W$.  Thus, $i(G) = i(G^*)$. 
	Furthermore, any $i$-set of $G$ is also an $i$-set of $G^*$, and so the $i$-sets of $G^*$ comprise of $W$ and the $i$-sets of $G$.  That is, $V(\ig{G^*}) = V(\ig{G}) \cup \{W\} = V(H) \cup \{W\}$.
	
	If $S$ is an $i$-set of $G$, then $S \cap W = \varnothing$.  Thus, $W$ is not adjacent to any other $i$-set in $\ig{G^*}$.  	 Relabelling the vertex representing the $i$-set $W$ in $G^*$ as $v$ in $\ig{G^*}$ yields  $\ig{G^* } = H \cup \{v\}$ as required.
	
	
	



	
	
	If $i(G)=1$, then $G$ has a dominating vertex; begin with $G \cup K_1$, which has $\ig{G} = \ig{G \cup K_1}$ and $i(G\cup K_1)=2$, and then proceed as above.	 
\end{proof}


\medskip

As a final lemma before demonstrating the $i$-graph realizability of forests, we show  that a pendant vertex can be added to any $i$-graph to create a new $i$-graph.

\begin{lemma} \label{lem:leafAdd}
	If $H$ is the $i$-graph of some graph $G$, and $H_u$ is the graph $H$ with some pendant vertex $u$ added, then there exists some graph $G_u$ such that $\ig{G_u} = H_u$.  
\end{lemma}

\begin{proof}
	By Lemma \ref{lem:epnAdd} we may assume that for any $i$-set $S$ of $G$, $\mathrm{epn}(v,S)\neq \varnothing$ for all $v \in S$. 	
 To construct $G_u$, begin with a copy of $G$.  
 If $w$ is the stem of $u$ in $H_u$, then consider the $i$-set $W=\{v_1,v_2,\dots,v_k\}$ in $G$ corresponding to $w$.  
 To each $v_i\in W$, attach a new vertex $x_i$ for all $1 \leq i \leq k$.  
 Then join each $x_i$ to a new vertex $y$, and then to $y$, add a final pendant vertex $z$.  
 Thus $V(G_u) = V(G) \cup \{x_1,x_2,\dots, x_k,y,z\}$ as in Figure \ref{fig:pendAdd}.

	\begin{figure}[H]
		\centering
		\begin{tikzpicture}	[scale=.8]				
			\node(cent) at (0,0) {};
			\path (cent) ++(0:20 mm) node (Wcent)  {};			
			\path (cent) ++(0:0 mm) node (Gcent)  {};	
			\path (cent) ++(0:40 mm) node (Xcent)  {};	
			\path (cent) ++(0:60 mm) node (Ycent)  {};
			
			\foreach \label/\rad in {1/0,2/1,3/2,4/3}
			{\path (Wcent) ++(90:\label*6 mm) node[std] (w\label)  {};				
				\path (w\label) ++(30: 4mm) node (wL\label)  {$v_{\label}$};
				
			}
			
			\foreach \label/\rad in {1/0,2/1,3/2,4/3}
			{
				\path (w\label) ++(0: 20mm) node[std] (x\label)  {};	
				\draw[thick] (w\label)--(x\label);
				
				\path (x\label) ++(150: 4mm) node (xL\label)  {$x_{\label}$};
			}
			
			\foreach \label/\rad in {1/0,2/1,3/2}
			{
				\path (Gcent) ++(90:4+6*\label mm) node[std] (g\label)  {};
				
			}
			
			\path (Ycent) ++(90: 15mm) node[std] (y)  {};	
			\path (y) ++(90: 5mm) node (yL)  {$y$};
			
			\foreach \label/\rad in {1/0,2/1,3/2,4/3}
			{
				\draw[thick] (y)--(x\label);
			}
			
			\path (y) ++(0: 15mm) node[std] (z)  {};	
			\path (z) ++(90: 5mm) node (zL)  {$z$};		
			\draw[thick](z)--(y);	
			
			\begin{pgfonlayer}{blob}			
				\draw[draw=mgray,  fill=white, dashed] \convexpath{w1,w4}{6mm};
				\path (w1) ++(-90: 8mm) node (WL)  {$W$};
				
				\draw[draw=mgray,  fill=white, dashed] \convexpath{g1,g3}{5mm};
				\path (g1) ++(-90: 6mm) node (GL)  {$G-W$};	
				
			\end{pgfonlayer}
			
			\begin{pgfonlayer}{bedge}

				\draw[draw=gray,  fill=lgray] \convexpath{w4,w1,g1,g3}{10mm};
				
			\end{pgfonlayer}

			\begin{pgfonlayer}{bedge}	
				\foreach \i in {1,2,3}
				\path (g\i) ++(0:5 mm) coordinate (gh\i)  {};
				\foreach \i in {1,2,3,4}	
				\path (w\i) ++(0:-5 mm) coordinate (wh\i)  {};

				\foreach \i in {1,2,3}			
				{
					\foreach \j in {1,2,3,4}	
					\draw[dashed, mgray] (gh\i)--(wh\j);
				}
			\end{pgfonlayer}
			
			\path (cent) ++(0:9 mm) coordinate (GLcent)  {};	
			\path (GLcent) ++(90:28 mm) node (GLcent)  {$G$};

		\end{tikzpicture}			
		\caption{The construction of $G_u$ from $G$  in Lemma \ref{lem:leafAdd}.}
		\label{fig:pendAdd}%
	\end{figure}
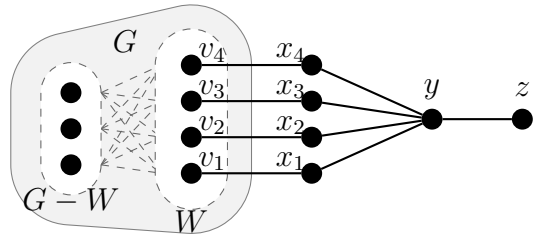
	
	It is easy to see that if $S$ is an $i$-set of $G$, then $S_y=S\cup\{y\}$ is an independent dominating set of $G_u$.  The set $W_z = W \cup \{z\}$ is also an independent dominating set of $G_u$.  Thus,  $i(G_u) \leq i(G)+1$.  It remains only to show that these are $i$-sets and the only $i$-sets of $G_u$.

	We claim that no $x_i$ in $X = \{x_1,x_2,\dots, x_k\}$ is in any $i$-set of $G_u$.  To show this, suppose to the contrary that $S^*$ is an $i$-set with $S^* \cap X = \{x_1,x_2,\dots, x_{\ell}\}$ for some $1 \leq \ell \leq k$.  Then, $y \notin S^*$; that is, $\{y,z\} \cap S^* = \{z\}$.  To dominate the remaining $\{x_{\ell+1},x_{\ell+2},\dots, x_k\}$, we have that $S^*=\{x_1,x_2,\dots,x_{\ell}\}\cup \{v_{\ell+1},v_{\ell+2},\dots, v_k\} \cup \{z\}$.  Recall from our initial assumption on $G$ that there exists some $v_1^* \in \mathrm{epn}_G(v_1,W)$.  Thus, $v_1^* \notin (N_{H_u}[\{v_{\ell+1},v_{\ell+2},\dots, v_k\}] \cap V(G))$, and so $v_1^*$ is undominated by $S^*$, which implies that $S^*$ is not an $i$-set.   
	
	Thus in every $i$-set of $G_u$, $y$ is dominated either by itself or by $z$.  If $y$ is not in a given $i$-set $S$ (and so $z\in S$), then to dominate $X$, $W\subseteq S$, and so $S = W \cup \{z\}$.  Conversely, if $y\in S$ (and $z\notin S$), then since the vertices of $G$ can only be dominated internally, $S$ is an $i$-set of $G_u$ if and only if $S-\{y\}$ is an $i$-set of $G$, which completes the proof of our claim.
	
	If  $u^*$ and $w^*$ are the vertices in $\ig{G_u}$ associated with $W_z$ and $W_y = W\cup \{y\}$ respectively, then clearly $\ig{G_u} - \{u^*\} \cong \ig{G}$.  Furthermore, since  $W_y$ is the only $i$-set  with $|W_y - W_z| = 1$ and $yz\in E(G_u)$, it follows that $\deg(u^*)=1$ and $u^*w^* \in E(\ig{G_u})$, and we conclude that $\ig{G_u} \cong H_u$.
\end{proof}

\medskip

\noindent Finally, we amalgamate the previous lemmas on adding isolated and pendant vertices to $i$-graphs to demonstrate that forests are $i$-graphs.

\begin{theorem} \label{thm:i:forests}
	All forests are  $i$-graph  realizable.
\end{theorem}

\begin{proof}
	We show by induction on the number of  vertices that if $F$ is a forest with $m$ components, then $F$ is $i$-graph realizable.  For a base, note that $\ig{\overline{K_2}} = K_1$.  Construct the graph  $\overline{K_m}$ by repeatedly applying Lemma \ref{lem:vAdd}.  Suppose  that all forests on $m$ components on at most $n$ vertices are $i$-graph realizable.  Let $F$ be some forest with  $|V(F)| = n+1$ and components $T_1,T_2,\dots, T_m$.   If all  vertices of $F$ are isolated, we are done, so assume there is some leaf $v$ with stem $w$ in component $T_1$.  Let $F^* =  F-\{v\}$.  By induction there exists some graph $G^*$ with $\ig{G^*} \cong F^*$.  Applying Lemma \ref{lem:leafAdd} to $G^*$ at $w$ constructs a graph $G$ with $\ig{G} \cong F$.
\end{proof}

\medskip

\noindent Moreover, by adding Proposition \ref{prop:i:cycle} to the previous results, we obtain the following immediate corollary.  

\begin{coro} \label{coro:unicyclic}
	Unicyclic graphs are $i$-graph realizable.
\end{coro}

With the completion of the constructions of forests and unicyclic graphs as $i$-graphs, we have now determined the $i$-graph realizability of many collections of small graphs.  In particular, we draw the reader's attention to the  following observation.  

\begin{obs} \label{obs:i:4v}  
	Every graph on at most four vertices except $\Dia$ is an $i$-graph.
\end{obs}

\section{Building $i$-Graphs} \label{sec:i:tools}

In this section, we examine how new $i$-graphs can be constructed from known ones.  We begin by presenting three very useful tools for constructing new $i$-graphs: the \textbf{Max Clique Replacement Lemma}, the \textbf{Deletion Lemma}, and the \textbf{Inflation Lemma}.  The first among these shows that maximal cliques in $i$-graphs can  be replaced by arbitrarily larger maximal cliques.

\begin{lemma}[Max Clique Replacement Lemma]  \label{lem:i:cliqueRep}
	Let $H$ be an $i$-graph with a maximal $m$-vertex clique, $\mathcal{K}_m$.  Then, the graph $H_w$ formed by adding a new vertex $w^*$ adjacent to all of $\mathcal{K}_m$ is also an $i$-graph.	
\end{lemma}

\begin{proof}	
	Suppose $G$ is a graph such that $\ig{G}=H$ and $i(G)=k+1$ where $k\geq 1$, and let $\mathcal{K}_m=\{V_1, V_2, \dots, V_m\}$ be a maximal clique in $H$.  From Lemma \ref{lem:i:claw}, the corresponding $i$-sets $V_1, V_2, \dots, V_m$ of $G$ differ on exactly one vertex, so for each $1\leq i \leq m$, let $V_i = \{v_i, z_1,z_2,\dots,z_{k}\} \subseteq V(G)$, so that $Z=\{z_1,z_2,\dots,z_{k}\} = \bigcap_{1 \leq i \leq m} V_i$.  Notice also from Lemma \ref{lem:i:claw}, for each $1 \leq i  < j \leq m$, $v_iv_j \in E(G)$, and so $Q_m=\{v_1,v_2,\dots,v_m\}$ is a (not necessarily maximal) clique of size $m$ in $G$. 
	
	In addition to ${Q_m}$ and $Z$ defined above, we further weakly partition (i.e. some of the sets of the partition may be empty) the vertices of $G$ as
	\begin{enumerate}[itemsep=0pt, label=]
		\item $X= N(Q_m) \backslash N(Z)$, the vertices dominated by $Q_m$ but not $Z$.
		\item $Y = N(Q_m) \cap N(Z)$, the vertices dominated by both $Q_m$ and $Z$.
		\item $A= N(Z) \backslash N(Q_m)$, the vertices dominated by $Z$ but not $Q_m$.
	\end{enumerate}

	\noindent This partition (as well as the construction of $G_w$ defined below) is illustrated in Figure \ref{fig:i:cliqueRep}.  Before proceeding with the construction, we state the following series of claims regarding the set $X$:
	\begin{enumerate}[itemsep=5pt, label=]
		\item \textbf{Claim 1:} \emph{Each  $x \in X$ is dominated by every vertex of $Q_m$.  }\\
		Otherwise, if some $x \in X$ is not adjacent to some $v_j \in Q_m$, then $x$ is undominated in the the $i$-set $V_j = \{v_j\} \cup Z$.

		\item \textbf{Claim 2:} \emph{$|X| \neq 1$}.\\ 
		If $X=1$, say $X = \{x\}$, then $X^*= \{x\}\cup Z$ is independent, dominating, and has $|X^*| = i(G)$; that is, $X^*$ is an $i$-set of $G$.  However, since $x$ is adjacent to all of $Q_m$ in $G$, $\edge{X^*,x,v_j,V_j}$ for each $1 \leq j \leq m$, contradicting the maximality of the clique $\mathcal{K}_m$ in $H$.
		
		\item \textbf{Claim 3:} \emph{No $x \in X$ dominates all of $X$}. \\ 
		If $x \in X$ dominates $X$, then $\{x\} \cup Z$ is an $i$-set of $G$. Following a similar argument  of  Claim 2, this contradicts the maximality of $\mathcal{K}_m$ in $H$.

		\item \textbf{Claim 4:} \emph{For any $v \in (X \cup Y  \cup A)$, $\{v\} \cup Z$ is not an $i$-set}.\\
		Combining Claims 2 and 3, if $v \in X$, then there exists some $x_i\in X$ such that $v\not\sim x_i$, and thus $\{v\} \cup Z$ does not dominate $x_i$.
		If $v \in (Y\cup A)$, then $v\in N(Z)$,  and so $\{v\} \cup Z$ is not independent.  		
		
	\end{enumerate}	
	
	We construct a new graph $G_w$ from $G$ by joining a new vertex $w$ to each vertex in $V(G)-Z$, as in Figure \ref{fig:i:cliqueRep}.  We claim that $\ig{G_w} \cong H_w$.  

	\begin{figure}[H]
		\centering
		\begin{tikzpicture}	
			
			\node(cent) at (0,0) {};
			\path (cent) ++(0:30 mm) node (Zcent)  {};
			\path (cent) ++(180:30 mm) node (Kcent)  {};
			
			\path (cent) ++(90:40 mm) node (Wcent)  {};
			
			\path (cent) ++(90: 0mm) node (Ycent)  {};		
			\path (cent) ++(180:60 mm) node (Xcent)  {};
			\path (cent) ++(0:70 mm) node (Vcent)  {};		
			
			\foreach \label/\rad in {1/6,2/2,3/-2,4/-6}
			\path (Kcent) ++(90:\rad mm) node[std] (k\label)  {};	
			
			\draw(k1)--(k4);
			\draw[black] (k1) to [out=225,in=135](k3);	
			\draw[black] (k2) to [out=225,in=135](k4);
			\draw[black] (k1) to [out=215,in=145](k4);

			\foreach \label/\rad in {1/6,2/0,3/-6}
			\path (Xcent) ++(90:\rad mm) node[std] (x\label)  {};
			
			\foreach \label/\rad in {1/6,2/-6}
			\path (Ycent) ++(90:\rad mm) node[std] (y\label)  {};
			
			\foreach \label/\rad in {1/0,2/-4,3/-8}
			\path (Zcent) ++(90:\rad mm) node[std] (z\label)  {};	
			
			\foreach \label/\rad in {1/6,2/2,3/-2, 4/-6}
			\path (Vcent) ++(90:\rad mm) node (v\label)[std]  {};	
			
			\path (cent) ++(120:30 mm) node (w)[regRed]  {};	
			\path (w) ++(90: 6mm) node (WL)  {$w$};	
			
			\draw[ll](x1)--(x2);
			\draw[ll](v1)--(v2)--(v3);	
			
			\begin{pgfonlayer}{blob}			
				\draw[draw=gray,  fill=lgray] \convexpath{k1,k4}{5mm};
				\path (k4) ++(-90: 15mm) node (KL)  {$Q_m$};
				
				\draw[draw=gray,  fill=lgray] \convexpath{x1,x3}{5mm};
				\path (x3) ++(-90: 15mm) node (XL)  {$X$};	
				
				\draw[draw=gray, thick, fill=lgray] \convexpath{z1,z3}{5mm};
				\path (z3) ++(-90: 13mm) node (ZL)  {$Z$};	
				
				\draw[draw=gray, thick, fill=lgray] \convexpath{y1,y2}{5mm};
				\path (y2) ++(-90: 15mm) node (YL)  {$Y$};	
				
				\draw[draw=gray, thick, fill=lgray] \convexpath{v1,v4}{5mm};
				\path (v4) ++(-90: 15mm) node (VL)  {$A$};				
				
			\end{pgfonlayer}
			
			
			\foreach \i in {1,2,3}
			{
				\foreach \j in {1,2,3,4}
				\draw[ll](x\i)--(k\j);
			}		
			
			\draw[ll](y1)--(k2);
			\draw[ll](y2)--(k4);
			\draw[ll](y1)--(k1);

			\draw[ll](z1)--(y1)--(z2);
			\draw[ll](y2)--(z2);

			\draw[ll](z1)--(v1);
			\draw[ll](z2)--(v2);
			\draw[ll](z2)--(v3);
			\draw[ll](z2)--(v4);
			\draw[ll](z3)--(v4);
			
			\draw[ll] (x3) to [out=-30,in=210](y2);	
			
			\draw[ll] (y2) to [out=-30,in=210](v4);	
			\draw[ll] (y2) to [out=-30,in=210](v3);

			

				\foreach \i in {1,2,3}
				\draw[red] (w)--(x\i);
				
				\foreach \i in {1,2,3,4}		
				{	\draw[red] (w)--(k\i);	
					\draw[red] (w)--(v\i);
				}
				
				\draw[red] (y1)--(w)--(y2);
			
		\end{tikzpicture}
		\caption{Construction of $G_w$ from $G$ in Lemma  \ref{lem:i:cliqueRep}.}
		\label{fig:i:cliqueRep}
	\end{figure}
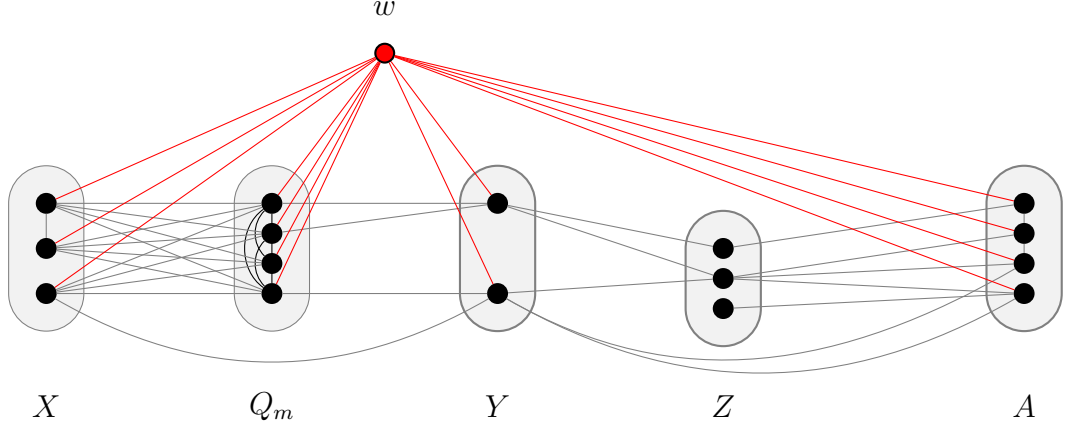	
	
	Let $S$ be some $i$-set of $G_w$.  If $w \notin S$, then $S \subseteq V(G)$ and so $S$ is also independent dominating in $G$, implying $|S| = i(G)= k+1$.  However, if $w\in S$, then since $w$ is adjacent to all of $V(G)-Z$ and $Z$ is independent, we have that $S = \{w\} \cup Z$.  It follows that $i(G_w) = i(G)$.  Moreover, any $i$-set of $G$ is also an $i$-set of $G_w$, and so $W:=\{w\} \cup Z$ is the only new $i$-set generated in $G_w$.  Thus $V(\ig{G_w}) = V(\ig{G}) \cup \{W\}$.
	
	Consider now the edges of $\ig{G_w}$.  Since $w$ is adjacent to all of $Q_m$ in $G_w$,  $\edge{W,w,v_j,V_j}$ for each $1 \leq j \leq m$, and thus $\mathcal{K}_m \cup W$ is a clique in $\ig{G_w}$. 
	
	Finally, we demonstrate that $W$ is adjacent only to the $i$-sets of $\mathcal{K}_m$.  Consider some $i$-set $S \notin \mathcal{K}_m$, and suppose to the contrary that $\onlyedge{W,S}$.  	
	As $W$ is the only $i$-set containing $w$, we have that $w\notin S$, and hence $\edge{W,w,u,S}$ for some vertex $u$.
	Since $w \sim u$,   $u \notin Z$.  Moreover, since $W$ and $S$ differ at exactly one vertex and $Z \subseteq W$, it  follows that $Z \subseteq S$; that is, $S = \{u\} \cup Z$.  If $ u\in Q_m$ then $S \in \mathcal{K}_m$, a contradiction.  If $u \in (X \cup Y \cup A)$, then by Claim 4, $S$ is not an $i$-set, which is again a contraction.  We conclude that $W \not\sim S$ for any $i$-set $S \notin \mathcal{K}_m$, and therefore $E(\ig{G_w}) = E(\ig{G}) \cup \left( \bigcup_{v_i\in Q_m} wv_i\right)$.  If follows that $\ig{G_w} \cong H_w$.		
\end{proof}

\medskip

Our next result, the Deletion Lemma, shows that the class of $i$-graphs is closed under vertex deletion. It is unique among our other constructions; unlike most of our results  which demonstrate how to build larger $i$-graphs from smaller ones, the Deletion Lemma instead shows that every induced subgraph of an $i$-graph is also an $i$-graph.

\begin{lemma}[The Deletion Lemma] \label{lem:i:inducedI} 
	If $H$ is a nontrivial $i$-graph, then any induced subgraph of $H$ is also an $i$-graph.
\end{lemma}

\begin{proof}  
	Let $G$ be a graph such that $H= \ig{G}$ and $i(G)=k$.   To prove this result, we show that for any $X \in V(H)$, there exists some graph $G_X$ such that $\ig{G_X} = H-X$.
	
	To construct $G_X$, take a copy of $G$ and add to it a vertex $z$ so that $z$ is adjacent to each vertex of $G-X$ (see Figure \ref{fig:i:Gx}).  	
	Observe first that since $H$ is nontrivial, there exists an $i$-set $S \neq X$ of $G$.  Then, $S$ is also an independent dominating set of $G_X$, and so $i(G_X) \leq k$.    		
	Consider now some $i$-set $S_X$ of $G_X$.   Clearly $S_X \neq X$ because $X$ does not dominate $z$.  	 
	If $z\in S_X$, then as $S_X$ is independent, no vertex of $G-X$ is in $S_X$.  
	Moreover, since $X$ is also independent and its vertices have all of their neighbors in $G-X$, this leaves each vertex of $X$ to dominate itself.  That is, $X \subseteq S_X$, implying that $S_X = X \cup \{z\}$ and  $|S_X| = k + 1$.  This contradicts that $i(G_X) \leq k$, and thus we conclude that $z$ is not in any $i$-set of $G_X$.  	  
	It follows that each $i$-set of $G_X$ is composed only of vertices from $G$ and so $i(G_X) = k$.  Thus, $S_X \neq X$ is an $i$-set of $G_X$ if and only if it is an $i$-set of $G$.   	 Given that $V(\ig{G_X}) = V(\ig{G}) - \{X\} = V(H) - \{X\}$, we have that $\ig{G_X} = H - X$ as required.  \null \hfill
\end{proof}


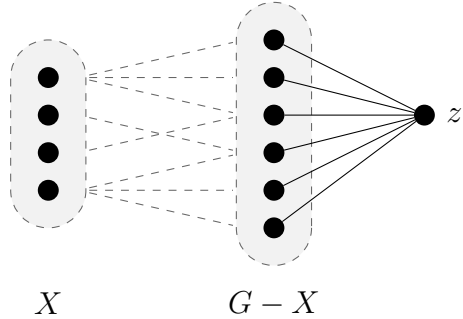
\begin{figure}[H]
	\centering
	\begin{tikzpicture}					
		
		\node(cent) at (0,0) {};
		
		\foreach \label/\rad in {1/0,2/5,3/10,4/15}
		\path (cent) ++(90:\rad mm) node (x\label)  {};		
		
		\path (cent) ++(0:30 mm) node (yM)  {};					
		
		\foreach \label/\rad in {1/-5,2/0,3/5,4/10,5/15,6/20}
		\path (yM) ++(90:\rad mm) node (y\label) [std]  {};		
		
		
		\foreach \i in {1,2,3,4}
		\path (x\i) ++(0:4 mm) node (h\i)  {};
		\foreach \i in {1,2,3,4,5,6}	
		\path (y\i) ++(0:-4 mm) node (j\i)  {};	
		
		\draw[dashed, mgray] (h1)--(j2);
		\draw[dashed, mgray] (h2)--(j4);
		\draw[dashed, mgray] (h4)--(j5);
		\draw[dashed, mgray] (h4)--(j4);
		\draw[dashed, mgray] (h3)--(j3);
		\draw[dashed, mgray] (h1)--(j3);	
		\draw[dashed, mgray] (h1)--(j1);	
		\draw[dashed, mgray] (h4)--(j6);	
		
		\draw[draw=mgray, fill=lgray, dashed] \convexpath{x1,x4}{5mm};	
		
		\foreach \i in {1,2,3,4}
		\node [std](xD\i) at (x\i) {};	
		
		\path (x1) ++(-90: 15mm) node (XL)  {$X$};		
		
		\draw[draw=mgray,fill=lgray, dashed] \convexpath{y1,y6}{5mm};		
		
		\path (y3) ++(90:5 mm) node (y35)  {};	
		\path (y35) ++(0:20 mm) node (z) [std] {};
		\path (z) ++(0:4 mm) node (zL) {$z$};
		
		\foreach \i in {1,2,3,4,5,6}{
			\node [std](yD\i) at (y\i) {};		
			\draw(z)--(yD\i);
		}
		
		\path (y1) ++(-90: 10mm) node (GL)  {$G-X$};	
		
	\end{tikzpicture}			
	\caption{Construction of $G_X$ in Lemma \ref{lem:i:inducedI}.}
	\label{fig:i:Gx}%
\end{figure}


\noindent The following corollary is immediate as the contrapositive of Lemma \ref{lem:i:inducedI}.  

\begin{coro} \label{coro:i:notInduced}
	If $H$ is not an $i$-graph, then any graph containing an induced copy of $H$ is also not an $i$-graph.
\end{coro}

\noindent  
This powerful corollary, although simple in statement and proof, immediately removes many families of graphs from $i$-graph realizability.  For example, all wheels, 2-trees, and maximal planar graphs on at least five vertices contain an induced copy of  the Diamond graph $\Dia$, which was shown in Proposition \ref{prop:i:diamond} to not be an $i$-graph.
Moreover, given that $i$-graph realizability is an inherited property, this suggests that there may be a finite-family forbidden subgraph characterization for $i$-graph realizability.

We now alter course to examine how one may construct new $i$-graphs by combining several known $i$-graphs.  Understandably, an immediate obstruction to combining the constructions of $i$-graphs of, say, $\ig{G_1} = H_1$ and $\ig{G_2}=H_2$ is that it is possible (and indeed, likely) that $i(G_1) \neq i(G_2)$.  

Two solutions to this quandary are presented in the following  lemmas. In the first, Lemma \ref{lem:i:infFam},  given a graph $G$, we  progressively construct an infinite family of seed graphs $\mathcal{G}$ with the same number of components as $G$, and such that $\ig{G} = \ig{G_j}$ for each $G_j \in \mathcal{G}$.    The second, Lemma \ref{lem:i:anySize} or the \textbf{Inflation Lemma}, offers a more direct solution: given an $i$-graph $H$, we demonstrate how to ``inflate" a seed graph $G$ to produce a new graph $G^*$ such that $\ig{G^*} = \ig{G}$ and the $i$-sets of $G^*$ are arbitrarily larger than the $i$-sets of $G$.


\begin{lemma} \label{lem:i:infFam}
	If $G$ is a graph with $\ig{G}\cong H$, then there exists an infinite family of graphs $\mathcal{G}$ such that $\ig{G_j}\cong H$ for each $G_j \in \mathcal{G}$. 
	Moreover, the number of components of $G_j \in \mathcal{G}$ is the same as $G$ ($k(G) = k(G_j)$).
	
\end{lemma}

\begin{proof}
	Suppose $v \in V(G)$, and let $G^*$ be the graph  obtained by attaching a copy of the star $K_{1,3}$ with $V(K_{1,3})=\{x,y_1,y_2,y_3\}$ ($\deg(x)=3$) by joining $v$ to $y_1$.   As $y_2$ and $y_3$ are pendant vertices, $i(G^*) \geq i(G)+1$.  If $S$ is an $i$-set of $G$, then $S^*=S \cup \{x\}$ is dominating and independent, and so $i(G^*) =  i(G)+1$.  Thus, $x$ is in every $i$-set of $G^*$, and we can conclude that $S^*$ is an $i$-set of $G^*$ if and only if $S^*-\{x\}$ is an $i$-set of $G$.		It follows that $\ig{G^*} \cong  \ig{G}$ as required.  		
	Attaching additional copies of $K_{1,3}$ as above at any vertex of $H$ similarly creates the other graphs of $\mathcal{G}$.		
\end{proof}

\begin{lemma}[Inflation Lemma] \label{lem:i:anySize} 
	If $H$ is the $i$-graph of some graph $G$, then for any $k \geq i(G)$ there exists a graph $G^*$ such that $i(G^*) =k$ and $\ig{G^*}\cong H$.  		
\end{lemma}

\begin{proof}
	Begin with a copy of $G$ and add to it $\ell = k - i(G)$ isolated vertices, $S=\{v_1,v_2,\dots,v_{\ell}\}$.   Immediately, $X$ is an $i$-set of $G$ if and only if $X\cup S$ is an $i$-set of $G^*$.  Moreover, if $X$ and $Y$ are $i$-sets of $G$ such that $\adjG{X,Y,G}$ in $H$, then $\adjG{( X\cup S), (Y\cup S), G^*}$, and so $\ig{G^*} \cong H$.  
\end{proof}

\medskip

Now, when attempting to combine the constructions of $\ig{G_1} = H_1$ and $\ig{G_2}=H_2$ and $i(G_1) < i(G_2)$, we need only inflate $G_1$ until its $i$-sets are the same size as those in $G_2$.  A powerful construction tool, the Inflation Lemma is used repeatedly in almost all of the following results of this section.

In the next result we show that, given $i$-graphs $H_1$ and $H_2$, a new $i$-graph $H$ can be formed by identifying any two vertices in $H_1$ and $H_2$. The proof here uses Proposition \ref{prop:i:cart}, the Deletion Lemma (Lemma \ref{lem:i:inducedI}), and the Inflation Lemma (Lemma \ref{lem:i:anySize}); a proof in which a source graph of $H$ is given can be found in \cite[Proposition 3.30]{LauraD}. This result provides an alternative proof for Theorem \ref{thm:i:forests}.

\begin{prop}\label{prop:i:cutV}	
	Let $H_1$ and $H_2$ be $i$-graphs.  Then the graph $H_{x=y}$, formed by identifying a vertex $x$ of $H_1$ with a vertex $y$ of $H_2$, is also an $i$-graph.		
\end{prop}

\begin{proof}
Suppose $G_1$ and $G_2$ are graphs such that $\ig{G_1}=H_1$ and $\ig{G_2} = H_2$.  Applying the Inflation Lemma we may assume that $i(G_1) = i(G_2)=k \geq 2$.  By Proposition \ref{prop:i:cart} there is a graph $G$ such that $\ig{G} = H_1 \boksie H_2$.  Since $H_{x=y}$ is an induced subgraph of  $H_1 \boksie H_2$,  we may apply the Deletion Lemma and delete all other vertices of $H_1 \boksie H_2$ until only $H_{x=y}$ remains.    
\end{proof}

\medskip

We use Proposition \ref{prop:i:cutV} to show that two $i$-graphs may be connected by an edge between any two vertices to produce a new $i$-graph.  A proof that gives a source graph for this new $i$-graph is given in \cite[Proposition 3.26]{LauraD}.

\begin{prop}\label{prop:i:brg}
	Let $H_1$ and $H_2$ be disjoint $i$-graphs.  Then the graph $H_{xy}$, formed by connecting  $H_1$ to $H_2$ by an edge between any $x\in V(H_1)$ and any $y \in V(H_2)$, is also an $i$-graph.			
\end{prop}

\begin{proof}
 Let $H_{3}\simeq K_{2}$ with $V(H_{3})=\{u,v\}$. Applying Proposition \ref{prop:i:cutV} twice,
we see that the graph $H_{xu}$ obtained by identifying $x\in VH_{1})$
with $u\in V(H_{3})$, and the graph $H_{xy}$ obtained by identifying $v\in
V(H_{xu})$ with $y\in V(H_{2})$ are $i$-graphs.   
\end{proof}

\medskip

The following corollary provides a way to connect two $i$-graphs with a clique rather than a bridge.  A constructive proof in which a source graph for the resulting $i$-graph is provided can be found in \cite[Corollary 3.27]{LauraD}.


\begin{coro}\label{coro:i:brg}
	Let $H_1$ and $H_2$ be $i$-graphs, and let $H$ be the graph formed from them as in Proposition \ref{prop:i:brg} by creating a bridge $xy$ between them.  Then the graph $H_m$  formed by replacing $xy$ with a $K_m$ for $m\geq 2$ is also an $i$-graph.
\end{coro}

\begin{proof}
Apply the Max Clique Replacement Lemma (Lemma \ref{lem:i:cliqueRep}) to the edge $xy$ in Proposition \ref{prop:i:brg}. 
\end{proof}

\medskip

The next proposition provides a method for combining two $i$-graphs without connecting them by an edge. 

\begin{prop} \label{prop:i:disjointU}
	If $H_1$ and $H_2$ are $i$-graphs, then $H_1 \cup H_2$ is an $i$-graph.		
\end{prop}

\begin{proof}
	Suppose $G_1$ and $G_2$ are graphs such that $\ig{G_1}=H_1$ and $\ig{G_2} = H_2$.  We assume that  $i(G_1) = i(G_2) \geq 2$.  Otherwise, apply the Inflation Lemma (Lemma \ref{lem:i:anySize}) to obtain graphs with $i$-sets of equal size at least 2.  Let $G=G_1 \vee G_2$, the join of $G_1$ and $G_2$. We claim that $\ig{G} = H_1 \cup H_2$.   
	
	We proceed similarly to the proof of Proposition \ref{prop:i:brg}; namely, if $S$ is an $i$-set of $G_1$, of $G_2$, then $S$ is an independent dominating set of $G$.  Likewise, we observe that any $i$-set of $G$ is a subset of $G_1$ or $G_2$, and so, $S$ is a $i$-set of $G$ if and only if it is an $i$-set of $G_1$ or $G_2$.  
	
	Suppose $\edgeG{X,x,y,Y,G_1}$.  Then in $G$, sets $X$ and $Y$ are still $i$-sets, and likewise, vertices $X$ and $Y$ are still adjacent, and so $\edgeG{X,x,y,Y,G}$.  Now suppose instead that $X$ is an $i$-set of $G_1$ and $Y$ is an $i$-set of $G_2$.  Within $G$, $X \cap Y = \varnothing$ and $|X|=|Y| \geq 2$, so $X$ and $Y$ are not adjacent in $\ig{G}$.  Therefore, $\adjG{X,Y,G}$ if and only if $\adjG{X,Y,G_1}$ or $\adjG{X,Y,G_2}$.  It follows that $\ig{G}= \ig{G_1} \cup \ig{G_2} = H_1\cup H_2$ as required.	
\end{proof}

\medskip

Applying these new tools in combination yields some unexpected results.  For example, the following corollary, which makes use of the previous Proposition \ref{prop:i:disjointU} in partnership with the Deletion Lemma (a construction for combining $i$-graphs and a construction for vertex deletions) gives our first result on $i$-graph edge deletions.

\begin{coro}\label{coro:i:cutBridge}	
	Let $H$ be an $i$-graph with a bridge $e$, such that the deletion of $e$ separates $H$ into components $H_1$ and $H_2$.  Then
	\begin{enumerate}[itemsep=1pt, label=(\roman*)]	
		\item \label{coro:i:cutBridge:1} $H_1$ and $H_2$ are $i$-graphs, and
		\item \label{coro:i:cutBridge:2} the  graph $H^*=H-e$ is an $i$-graph. 	
	\end{enumerate}
\end{coro}

\begin{proof}
	Part \ref{coro:i:cutBridge:1} follows immediately from Lemma \ref{lem:i:inducedI}.  
	For \ref{coro:i:cutBridge:2}, by Part \ref{coro:i:cutBridge:1}, $H_1$ and $H_2$ are $i$-graphs.  Proposition \ref{prop:i:disjointU} now implies that $H_1 \cup H_2 = H-e$ is also an $i$-graph. 
\end{proof}

\medskip

Combining the results of Proposition \ref{prop:i:cutV}	with Proposition \ref{prop:i:brg} and Corollary \ref{coro:i:cutBridge}, yields the following main result.

\begin{theorem}\label{thm:i:iffBlock}
	A graph $G$ is an $i$-graph if and only if all of its blocks are $i$-graphs.
\end{theorem}

\noindent As observed in Corollary \ref{coro:i:notInduced}, graphs with an induced $\Dia$ subgraph are not $i$-realizable.  If we consider the family of connected chordal graphs excluding those with an induced copy of  $\Dia$, we are left with the family of block graphs (also called clique trees): graphs where each block is a clique.  As cliques are their own $i$-graph, the following is immediate.

\begin{prop} \label{prop:i:block}
	Block graphs are $i$-graph realizable.  
\end{prop}

Cacti are  graphs whose blocks are cycles or edges.  Thus, we have the following immediate corollary.


\begin{coro} \label{coro:i:cacti}	
	Cactus graphs are $i$-graph realizable.  
\end{coro}

While the proof of Proposition \ref{prop:i:cutV} does provide a method for building block graphs, it is laborious to do so on a graph with many blocks, as the construction is iterative, with each block being appended one at a time. However, when we consider that the blocks of block graphs are complete graphs, and that complete graphs are their own $i$-graphs (and thus arguably the easiest $i$-graphs to construct), it is logical that there is a simpler construction.  We offer one such construction below.  An example of this process is illustrated in Figure~\ref{fig:i:block}.

\begin{cons} \label{cons:i:block} Let $H$ be a block graph with $V(H) = \{v_1,v_2,\dots, v_n\}$ and let  $\mathcal{B}_H = \{B_1, B_2, \dots, B_m\}$ be the collection of maximal cliques of $H$.  	To construct a graph $G$ such that $\ig{G}=H$:
	
	\begin{enumerate}[itemsep=1pt, label=(\roman*)]	
		\item Begin with a copy of each of the maximal cliques of $H$, labelled $A_1, A_2, \dots, A_m$ in $G$, where $A_i$ of $G$ corresponds to $B_i$ of $H$ for each $1 \leq i \leq m$, and the $A_i$ are pairwise disjoint. Notice that each cut vertex of $H$ has multiple corresponding vertices in $G$.
		
		\item \label{item:cons:i:block:cutV} Let $v \in V(H)$ be a cut vertex and $\mathcal{B}_v$ be the collection of blocks containing $v$ in $H$; for notational ease, say  $\mathcal{B}_v= \{B_1,B_2,\dots,B_k\}$, and suppose that $W=\{w_1, w_2,\dots, w_k\} \subseteq V(G)$ are the $k$ vertices corresponding to $v$, where $w_i \in A_i$ for all $1 \leq i \leq k$.  
		
		For each distinct pair $w_i$ and $w_j$ of $W$, add to $G$ three internally disjoint paths of length two between $w_i$ and $w_j$.  Since $v$ is in $k$ blocks of $H$, $3\binom{k}{2}$ vertices are added in this process.  These additions are represented as the green vertices in Figure \ref{fig:i:block}.	
		
		\item Repeat Step \ref{item:cons:i:block:cutV} for each cut vertex of $H$.
	\end{enumerate}	
\end{cons}	


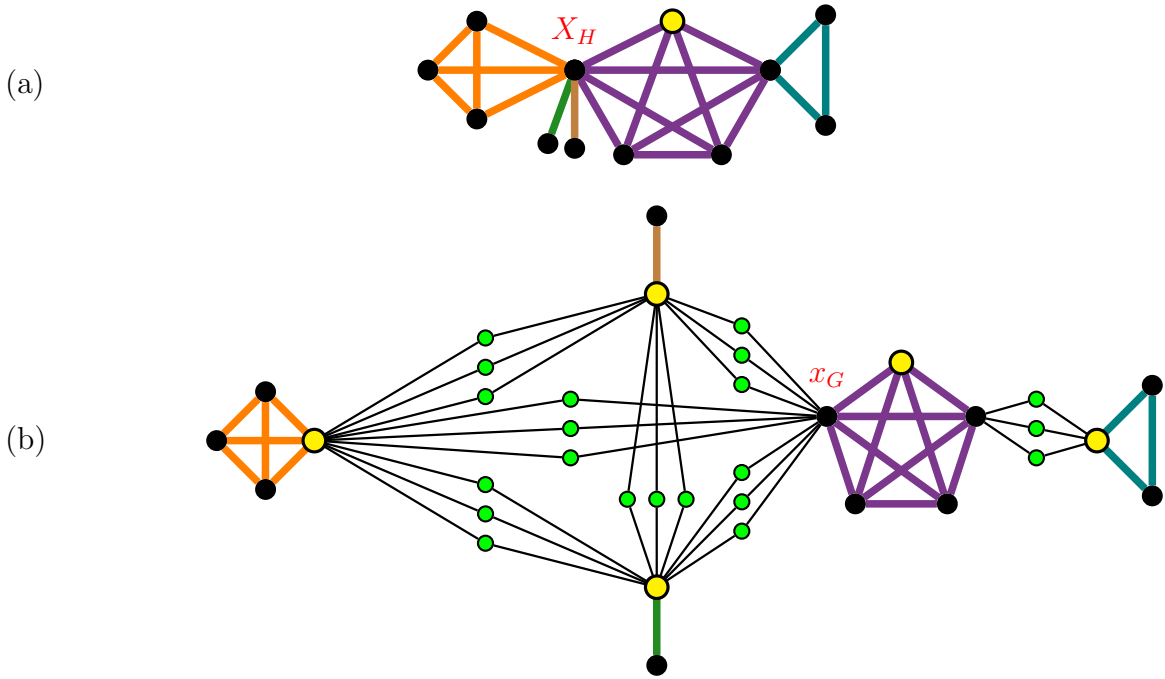
\begin{figure}[H]
	\begin{minipage}{0.1\textwidth}
		(a)
	\end{minipage}	
	\begin{minipage}{0.8\textwidth} \centering
		\begin{tikzpicture}[line width=0.3mm,scale=1.3]			
			
			\node(cent) at (0,0) {};	
			
			\path (cent) ++(0:10 mm) node (centR)  {};
			\path (cent) ++(0:-10 mm) node (centL)  {};
			
			\path (centL) ++(0:-10 mm) node (centA)  {};
			\path (centL) ++(0: 0 mm) node (centB)  {};
			\path (centL) ++(0: 0 mm) node (centD)  {};
			\path (cent) ++(0: 0 mm) node (centE)  {};
			\path (centR) ++(0: 0 mm) node (centF)  {};
			
			\foreach \label/\j in {2/90,3/180, 4/270}
			{	\path (centA) ++(\j:5 mm) node[std](a\label)  {};	
			}
			\path (centL) ++(0:0 mm) node[std](a1)  {};	
			\draw(a1)--(a2)--(a3)--(a4)--(a1)--(a3);
			\draw(a2)--(a4);
			
			\foreach \label/\j in {1/0,2/8}
			{	\path (centB) ++(270:\j mm) node[std](b\label)  {};	
			}
			\draw(b1)--(b2);			
			
			\foreach \label/\j in {1/0,2/8}
			{	\path (centD) ++(250:\j mm) node[std](d\label)  {};	
			}
			\draw(d1)--(d2);	
			
			\foreach \label/\j/\k in {2/90/5, 4/240/10, 5/300/10}
			{	\path (centE) ++(\j:\k mm) node[std](e\label)  {};	
			}
			
			\path (centL) ++(0:0 mm) node[std](e3)  {};	
			\path (e3) ++(90:4 mm) node (el3)  {\color{red}$X_H$};	
			\path (centR) ++(0:0 mm) node[std](e1)  {};	
			
			\draw(e1)--(e2)--(e3)--(e4)--(e5)--(e1)--(e3)--(e5)--(e2)--(e4)--(e1);	
			
			\foreach \label/\j in {2/45,3/-45}
			{	\path (centF) ++(\j:8 mm) node[std](f\label)  {};	
			}
			\path (centR) ++(0:0 mm) node[std](f1)  {};	
			\draw(f1)--(f2)--(f3)--(f1);	
			
			{
				\draw[ltteal, line width=1mm](f1)--(f2)--(f3)--(f1);
				\draw[medOrch, line width=1mm](e1)--(e2)--(e3)--(e4)--(e5)--(e1)--(e3)--(e5)--(e2)--(e4)--(e1);	
				\draw[forGreen, line width=1mm](d1)--(d2);
				\draw[brown, line width=1mm](b1)--(b2);
				\draw[orange, line width=1mm](a1)--(a2)--(a3)--(a4)--(a1)--(a3);
				\draw[orange, line width=1mm](a2)--(a4);
				
			}		
			
			{
				\node[byellow](s1) at (e2) {};	
			}
		\end{tikzpicture}			
	\end{minipage}	

\vspace{5mm}

	\begin{minipage}{0.1\textwidth}
		(b)
	\end{minipage}	
	\begin{minipage}{0.9\textwidth} \centering
		\begin{tikzpicture}[line width=0.3mm,scale=1.3]			
			
			\node(cent) at (0,0) {};	
			
			\path (cent) ++(0:10 mm) node (centR)  {};
			\path (cent) ++(0:-10 mm) node (centL)  {};
			
			\path (centL) ++(0:-55 mm) node (centA)  {};
			\path (centL) ++(0: -15 mm) node (centB)  {};
			\path (centL) ++(0: -15 mm) node (centD)  {};
			\path (cent) ++(0: 0 mm) node (centE)  {};
			\path (centR) ++(0: 0 mm) node (centF)  {};

			{	
				\foreach \label/\j in {2/90,3/180, 4/270, 1/0}
				{	\path (centA) ++(\j:5 mm) node[std](a\label)  {};	
				}
				\draw(a1)--(a2)--(a3)--(a4)--(a1)--(a3);
				\draw(a2)--(a4);
				
				\foreach \label/\j in {1/15,2/23}
				{	\path (centB) ++(90:\j mm) node[std](b\label)  {};	
				}
				\draw(b1)--(b2);

				\foreach \label/\j in {1/15,2/23}
				{	\path (centD) ++(270:\j mm) node[std](d\label)  {};	
				}
				\draw(d1)--(d2);	
				
				\foreach \label/\j/\k in {2/90, 3/162, 4/234, 5/306, 1/18}
				{	\path (centE) ++(\j:8 mm) node[std](e\label)  {};	
				}
				
				\path (e3) ++(90:4 mm) node (el3)  {\color{red}$x_G$};	
				
				
				\draw(e1)--(e2)--(e3)--(e4)--(e5)--(e1)--(e3)--(e5)--(e2)--(e4)--(e1);	
				
				\path (centR) ++(0:10 mm) node[std](f1)  {};	
				
				\foreach \label/\j in {2/45,3/-45}
				{	\path (f1) ++(\j:8 mm) node[std](f\label)  {};	
				}
				
				\draw(f1)--(f2)--(f3)--(f1);	
				
				\draw[ltteal, line width=1mm](f1)--(f2)--(f3)--(f1);
				\draw[medOrch, line width=1mm](e1)--(e2)--(e3)--(e4)--(e5)--(e1)--(e3)--(e5)--(e2)--(e4)--(e1);	
				\draw[forGreen, line width=1mm](d1)--(d2);
				\draw[brown, line width=1mm](b1)--(b2);
				\draw[orange, line width=1mm](a1)--(a2)--(a3)--(a4)--(a1)--(a3);
				\draw[orange, line width=1mm](a2)--(a4);				
			}		
			
			{
				\node[regG] (ef) at ($(f1)!0.5!(e1)$) {};
				\draw(e1)--(ef)--(f1);
				\foreach \label/\j in {1/90,2/-90}
				{
					\path (ef) ++(\j:3 mm) node[regG](ef\label)  {};
					\draw(e1)--(ef\label)--(f1);
				}
				
				\node[regG] (ab) at ($(a1)!0.5!(b1)$) {};
				
				\draw(a1)--(ab)--(b1);
				\foreach \label/\j in {1/90,2/-90}
				{
					\path (ab) ++(\j:3 mm) node[regG](ab\label)  {};
					\draw(a1)--(ab\label)--(b1);
				} 	
				
				\node[regG] (ad) at ($(a1)!0.5!(d1)$) {};
				
				\draw(a1)--(ad)--(d1);
				\foreach \label/\j in {1/90,2/-90}
				{
					\path (ad) ++(\j:3 mm) node[regG](ad\label)  {};
					\draw(a1)--(ad\label)--(d1);
				} 			 		
				
				\node[regG] (ae) at ($(a1)!0.5!(e3)$) {};
				
				\draw(a1)--(ae)--(e3);
				\foreach \label/\j in {1/90,2/-90}
				{
					\path (ae) ++(\j:3 mm) node[regG](ae\label)  {};
					\draw(a1)--(ae\label)--(e3);
				} 		
				
				\node[regG] (de) at ($(d1)!0.5!(e3)$) {};
				
				\draw(d1)--(de)--(e3);
				\foreach \label/\j in {1/90,2/-90}
				{
					\path (de) ++(\j:3 mm) node[regG](de\label)  {};
					\draw(d1)--(de\label)--(e3);
				} 	
				\node[regG] (be) at ($(b1)!0.5!(e3)$) {};
				
				\draw(b1)--(be)--(e3);
				\foreach \label/\j in {1/90,2/-90}
				{
					\path (be) ++(\j:3 mm) node[regG](be\label)  {};
					\draw(b1)--(be\label)--(e3);
				} 		
				
				
				\node[regG] (bd) at ($(b1)!0.7!(d1)$) {};
				\draw(b1)--(bd)--(d1);
				\foreach \label/\j in {1/0,2/180}
				{
					\path (bd) ++(\j:3 mm) node[regG](bd\label)  {};
					\draw(b1)--(bd\label)--(d1);
				} 					
				
			}
			
			{
				\node[byellow] (se) at (e2){};
				\node[byellow] (sa) at (a1){};
				\node[byellow] (sb) at (b1){};
				\node[byellow] (sd) at (d1){};
				\node[byellow] (sf) at (f1){};
				
			}	  
		\end{tikzpicture}		
	\end{minipage}		
	\caption{The construction of $G$ from $H$  in the proof of Proposition \ref{prop:i:block}.}
	\label{fig:i:block}%
\end{figure}

To see that the graph $G$ from Construction \ref{cons:i:block}  does indeed have $\ig{G}=H$, notice that $i(G)=m$, where $m$ is the number of blocks in $H$; if $X$ is an $i$-set of $G$, then $|X \cap A_i|$=1 for each $A_i \in \{A_1, A_2, \dots A_m\}$.  Moreover, no $i$-set of $G$ has vertices in the added green vertices, because, as with the proof of Proposition \ref{prop:i:cutV}, the inclusion of any one of these green vertices in an independent dominating set necessitates the addition of them all. 

In Figure \ref{fig:i:block}(b), the five yellow vertices form the $i$-set corresponding to the yellow vertex of $G$ in Figure \ref{fig:i:block}(a).  Only the token on the purple $K_5$ can move in $G$; the other four tokens remain frozen, thereby generating the corresponding purple $K_5$ of $H$.  It is only when the token on the purple $K_5$ is moved to the vertex $x_G$ that the tokens on the orange $K_4$, and the brown and green $K_2$'s, unfreeze one clique at a time.  This corresponds to the cut vertex $i$-set $X_H$ of $H$.  The freedom of movement now transfers from the purple  $K_5$ to any of the three other cliques, allowing for the generation of their associated blocks in  $G$ as required.


Finally, before we depart from block graphs, as chordal graphs are among the most well-studied families of graphs, we offer one additional reframing of this block graph result from the chordal graph perspective.   

\begin{coro} A chordal graph is $i$-graph realizable if and only if it is $\Dia$-free.
\end{coro}

With the addition of Proposition \ref{prop:i:block} to the results used to build Observation \ref{obs:i:4v}, this leaves only the house graph (see Figure \ref{fig:i:houseConstruct}(b)) as unsettled with regard to its $i$-graph realizability among the 34 non-isomorphic graphs on five vertices.  Although not strictly a result concerning the construction of larger $i$-graphs from known results, we include the following short proposition here for the sake of completeness.

\begin{prop}\label{prop:i:house}
	The house graph $\house$ is an $i$-graph.
\end{prop}

To demonstrate Proposition \ref{prop:i:house}, we provide an exact seed graph for the $i$-graph:  the graph $G$ in Figure  \ref{fig:i:houseConstruct}(a) ($K_3$ with a $P_3$ tail) has  $\ig{G} =\house$.  The $i$-sets of $G$ and their adjacency are overlaid on $\mathcal{H}$ in Figure  \ref{fig:i:houseConstruct}(b).

\begin{figure}[H]
	\centering	
	\begin{subfigure}{.4\textwidth}  \centering
		\begin{tikzpicture}			
			\coordinate (cent) at (0,0) {};
			
			\foreach \i / \j in {a/1,b/2,c/3}
			{
				\path(cent) ++(0: \j*10mm) node[std,label={ 90:${\i}$}]  (\i) {};
			}
			\draw[thick] (a)--(c);	
			
			\foreach \i / \j in {d/45,e/-45}
			{
				\path(c) ++(\j: 15 mm) node[std,label={ 2*\j:${\i}$}]  (\i) {};
			}
			
			\draw[thick](c)--(d)--(e)--(c);	
			
			\path (cent) ++ (90:30mm) coordinate (vfil) {};
			
			
			
		\end{tikzpicture}
		\caption{A graph $G$ such that $\ig{G} = \mathcal{H}$.}
		\label{fig:i:houseConstruct:a}%
	\end{subfigure}	\hspace{1cm}
	\begin{subfigure}{.5\textwidth} \centering
		\begin{tikzpicture}			
			\coordinate (cent) at (0,0) {};
			
			\foreach \i / \j / \k in {1/a/e, 2/a/d, 3/b/d, 4/b/e}
			{
				\path(cent) ++(\i*90-45: 13mm) node[std,label={\i*90-45:$\{\j,\k\}$}]  (v\i) {};
			}
			\draw[thick] (v1)--(v2)--(v3)--(v4)--(v1);
			
			\path(cent) ++(90: 20mm) node[std,label={ 90:$\{a,c\}$}]  (v5) {};	
			
			\draw[thick] (v1)--(v5)--(v2);
			
		\end{tikzpicture}
		\caption{The house graph $\mathcal{H}$ with $i$-sets of $G$.}
		\label{fig:i:houseConstruct:b}%
	\end{subfigure}		
	\caption{The graph $G$ for Proposition \ref{prop:i:house} with $\ig{G}=\mathcal{H}$.}
	\label{fig:i:houseConstruct}%
\end{figure}
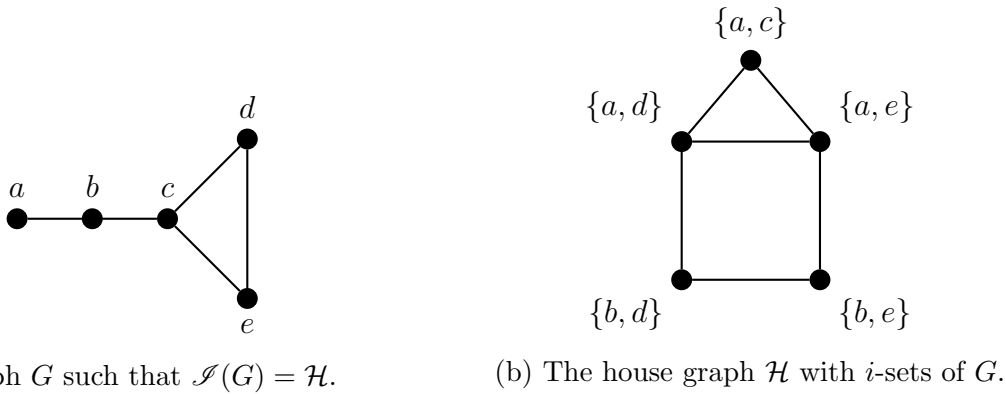



\section{Conclusion}
As we observed above, although not every graph is $i$-graph realizable, every graph does have an $i$-graph.  The exact structure of the resulting $i$-graph can vary among families of graphs from the simplest isolated vertex to surprisingly complex structures.  To illustrate this point, we determine the $i$-graphs of paths and cycles in \cite{BMT2}. 

We showed in Section \ref{sec:i:Real} that the theta graphs $\Dia \cong \thet{1,2,2}$, $K_{2,3}\cong\thet{2,2,2}$, and $\kappa \cong \thet{2,2,3}$ are not $i$-graph realizable. In \cite{BMT3} we investigate the class of theta graphs and determine exactly which ones fail to be $i$-graph realizable -- there are only finitely many such graphs. We also present a graph that is neither a theta graph nor $i$-graph realizable. The following question remains open.

\begin{ques}
Does there exist a finite forbidden subgraph characterization of $i$-graph realizable graphs?
\end{ques}

\medskip

\noindent\textbf{Acknowledgement\hspace{0.1in}}We acknowledge the support of
the Natural Sciences and Engineering Research Council of Canada (NSERC), RGPIN-2014-04760 and RGPIN-03930-2020.

\noindent Cette recherche a \'{e}t\'{e} financ\'{e}e par le Conseil de
recherches en sciences naturelles et en g\'{e}nie du Canada (CRSNG), RGPIN-2014-04760 and
RGPIN-03930-2020.
\begin{center}
\includegraphics[width=2.5cm]{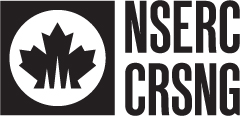}%
\end{center}
	

	\bibliographystyle{abbrv} 
	\bibliography{LTPhDBib} 
	
\end{document}